\documentclass[12pt]{article}

\usepackage[T1]{fontenc}
\usepackage[english]{babel}
\usepackage{amssymb}
\usepackage{amsmath}
\usepackage{epsfig}
\usepackage{color}
\usepackage{bbm}
\usepackage{graphicx}
\usepackage[a4paper]{geometry}
\allowdisplaybreaks

\newcommand\R{\mathbb{R}}
\def\dsp{\displaystyle}
\def\F{\mathcal F}
\newcommand\1{\mathbbm{1}}

\def\kxi{\begin{pmatrix}1\\ \xi\end{pmatrix}}

\newtheorem{theorem}{Theorem}[section]
\newtheorem{lemma}[theorem]{Lemma}
\newtheorem{proposition}[theorem]{Proposition}
\newtheorem{corollary}[theorem]{Corollary}

\newenvironment{proof}[1][Proof]{\begin{trivlist}
\item[\hskip \labelsep {\bfseries #1}] \it }{$\blacksquare$\end{trivlist}}

\newtheorem{remark}[theorem]{Remark}
\graphicspath{{Figures/}{}}

\definecolor{Red}{rgb}{1,0,0}
\definecolor{Blue}{rgb}{0,0,1}
\definecolor{Green}{rgb}{0,1,0}

\definecolor{red}{rgb}{1,0,0}
\definecolor{blue}{rgb}{0,0,1}

\begin{document}

\title{Kinetic entropy for the layer-averaged hydrostatic Navier-Stokes equations}

\author{E.~Audusse\footnote{Universit\'e Paris 13, Laboratoire d'Analyse, G\'eom\'etrie et Applications, 
99 av. J.-B. Cl\'ement, F-93430 Villetaneuse, France},
M.-O. Bristeau\footnote{Inria Paris, 2 rue Simone Iff, CS 42112, 75589
  Paris Cedex 12, France}$\;^{,}$\footnote{Sorbonne Universit\'es UPMC Univ Paris 6, Laboratoire Jacques-Louis Lions, 4 Place Jussieu, 75252 Paris cedex 05, France}$\;^{,}$\footnote{CEREMA, 134 rue de Beauvais, F-60280
Margny-L\`es-Compi\`egne, France}$\;$ \&
  J. Sainte-Marie\footnotemark[2]$\;^{,}$\footnotemark[3]$\;^{,}$\footnotemark[4]$\;^{,}$\footnote{Corresponding author: Jacques.Sainte-Marie@inria.fr}}
\date{\today}

\maketitle
\begin{abstract}
We are interested in the numerical approximation of the hydrostatic
free surface incompressible Navier-Stokes equations. By using a
layer-averaged version of the equations, we are able to extend
previous results obtained for shallow water system. We derive a
vertically implicit / horizontally explicit finite volume kinetic
scheme that ensures the positivity of the approximated water depth,
the well-balancing and a fully discrete energy inequality.
\end{abstract}

{\it Keywords~:} Incompressible Euler and Navier-Stokes sytems, free
surface flows, layer-averaged model, finite volumes, kinetic solver,
hydrostatic reconstruction, discrete entropy inequality, IMEX scheme.
\smallskip

{\it 2000 Mathematics Subject Classification: 65M12, 74S10, 76M12, 35L65.}

\tableofcontents

\section{Introduction}
Shallow water equations \cite{bouchut_book,bristeau1} have been widely used to model free surface geophysical fluid flows. This hyperbolic system can be derived from free surface incompressible Navier-Stokes equations by integration along the vertical direction and under a long wave approximation, that implies in particular an hydrostatic distribution of the pressure at the leading order \cite{gerbeau}. Due to this reduction of dimension, it is much more easy to deal with in a numerical point of view. Nevertheless it is not suitable for certain situations - stratified flows, wind-driven vertical circulation... In these cases, it is necessary to come back to three dimensional models but the hydrostatic assumption remains mostly valid, leading to consider the hydrostatic Navier-Stokes equations, also known as primitive equations \cite{brenier,grenier,azerad,temam,masmoudi}, see also \cite{pedlosky, olbers} for a general introduction to ocean models. In a previous work \cite{JSM_M2AN}, we introduced a layer-averaged approach to deal with this hydrostatic Navier-Stokes system in a framework that shares (forgetting for a while the viscous part to concentrate on Euler equations and advective processes) some hyperbolic properties with the shallow water equations. Our main result in this work is to propose a vertically implicit and horizontally explicit colocated finite volume scheme to compute approximate solutions of this layer-averaged model for which we are able to prove positivity of the water depth and a fully discrete energy inequality.\\

It is well known that  
incompressible hydrostatic Euler and shallow water equations 
satisfy some invariant domain properties since the water depth of the flow remains nonnegative.
Moreover regular solutions satisfy an energy equality.
Extending the theory developed for scalar hyperbolic equations, 
one assumes the energy associated to discontinuous solution, 
that acts as a mathematical entropy,
has to decrease through the shocks. 
To extend these properties to the numerical approximation is not
straightforward. A huge literature was devoted to derive explicit
colocated finite volume scheme for the shallow water equations
including topography source terms, but, up to our knowledge, only very
few schemes are endowed with these properties
\cite{simeoni,bouchut_book}. In a recent work \cite{JSM_entro}, a
kinetic framework was used to prove that the hydrostatic
reconstruction technique~\cite{bristeau1} associated to a kinetic
scheme \cite{Perthame90,perthame,bristeau} is able to provide, under a
classical CFL condition,  approximate solutions of the shallow water
system with topography source terms that are positive and satisfy a
fully discrete entropy inequality with a right hand side, or let say
an error term, that is proved to be proportional to $\Delta x^2$ with
a constant that is independent of the regularity of the solution -
that may develop discontinuities since one deals with hyperbolic
system. This inequality is the key step to prove the convergence of
the scheme \cite{lhebrard}. Thanks to the similarity between the
shallow water equations and the layer-averaged model for hydrostatic
Euler equations, it is possible to extend the kinetic framework and
the hydrostatic reconstruction technique, and then parts of the
previous proof, to the latter. But the new terms related to the
vertical direction makes the things more intricate, and the fully
explicit finite volume scheme originally proposed in \cite{JSM_M2AN}
may suffer from a very restrictive CFL condition in some
situations. In this work, we then propose a new implicit
discretization of the exchange terms in the vertical direction that
allows us to obtain entropy satisfying approximate solutions under the
same CFL condition as for the shallow water system. Up to our
knowledge, it is the first time this kind of result is obtained in
this two dimensional $(x,z)$ context. From a practical point of view, the linear problem to solve in the implicit part is restricted to the vertical direction and its size is then proportionnal to the number of layers that are introduced in the model, independently of the horizontal discretization. The added CPU cost is then strictly limited. \\

The outline of the paper is as follows : in Section \ref{sec:NS}, we recall the free surface incompressible hydrostatic Navier-Stokes equations ; in Section \ref{sec:av_euler}, we consider the inviscid Euler version of the equations and introduce the layer-averaged model and its kinetic description and we derive the related IMEX kinetic scheme ; in Section \ref{sec:properties}, we demonstrate the discrete energy inequality, first, on flat bottom and, second, including a bottom topography ; finally, in Section \ref{sec:lans}
we extend the results to the Navier-Stokes case.


\section{The Navier-Stokes system}
\label{sec:NS}

We begin by considering the two-dimensional hydrostatic incompressible Navier--Stokes system \cite{lions} describing a free surface gravitational flow moving over a bottom
topography $z_b(x)$. For free surface flows, the hydrostatic assumption consists in
neglecting the vertical acceleration, see \cite{brenier,grenier,azerad,temam,masmoudi} for justifications and mathematical analysis of the obtained models.

\subsection{The hydrostatic incompressible Navier-Stokes system}

We denote with $x$ and $z$ the horizontal and vertical directions, respectively.  
The system reads
\begin{eqnarray}
& & \frac{\partial u}{\partial x}+\frac{\partial w}{\partial z} =0,\label{eq:NS_2d1}\\
& & \frac{\partial u}{\partial t } + \frac{\partial u^2}{\partial x } +\frac{\partial uw }{\partial z }+ \frac{\partial p}{\partial x } = \frac{\partial \Sigma_{xx}}{\partial x} + \frac{\partial \Sigma_{xz}}{\partial z},\label{eq:NS_2d2}\\
& & 
\hspace*{3cm} \frac{\partial p}{\partial z} = -g + \frac{\partial \Sigma_{zx}}{\partial x} + \frac{\partial \Sigma_{zz}}{\partial z},
\label{eq:NS_2d3}
\end{eqnarray}
and we consider solutions of the equations  for
$$t>t_0, \quad x \in \R, \quad z_b(x) \leq z \leq \eta(t,x):=h(t,x)+z_b(x),$$
where $\eta(t,x)$ represents the free surface elevation, $h(t,x)$ the
water depth, ${\bf u}=(u,w)^T$ the velocity
vector and $g$ the gravity acceleration. 

The chosen form of the viscosity tensor is
\begin{eqnarray*}
\Sigma_{xx} = 2 \mu \frac{\partial u}{\partial x}, & & \Sigma_{xz} = \mu \bigl( \frac{\partial u}{\partial z} + \frac{\partial w}{\partial x} \bigr),\label{eq:visco1}\\
\Sigma_{zz} = 2 \mu \frac{\partial w}{\partial z}, && \Sigma_{zx} = \mu \bigl( \frac{\partial u}{\partial z} + \frac{\partial w}{\partial x}\bigr),
\label{eq:visco2}
\end{eqnarray*}
where $\mu$ is a dynamic  viscosity.

\subsection{Boundary conditions}

The system (\ref{eq:NS_2d1})-(\ref{eq:NS_2d3}) is completed with boundary conditions 
at the bottom and at the free surface. 
The outward unit normal vector to the free surface ${\bf n}_s$ and the upward
unit normal vector to the bottom ${\bf n}_b$ are given by
$${\bf n}_s = \frac{1}{\sqrt{1 + \bigl(\frac{\partial \eta}{\partial x}\bigr)^2}}
  \left(\begin{array}{c} -\frac{\partial \eta}{\partial x}\\ 1 \end{array} \right), 
  \quad {\bf n}_b = 
  \frac{1}{\sqrt{1 + \bigl(\frac{\partial z_b}{\partial x}\bigr)^2}} 
   \left(\begin{array}{c} -\frac{\partial z_b}{\partial x}\\ 1 \end{array} \right),$$
 respectively.  
We then denote with $\Sigma_T$  the total stress tensor, which has the form:
$$\Sigma_T = - p I_d + \left(\begin{array}{cc} \Sigma_{xx} & \Sigma_{xz}\\ 
\Sigma_{zx} & \Sigma_{zz}\end{array}\right).$$

\subsubsection{Free surface conditions}

At the free surface we have the kinematic boundary condition
\begin{equation}
\frac{\partial \eta}{\partial t} + u_s \frac{\partial \eta}{\partial x}
-w_s = 0,
\label{eq:free_surf} 
\end{equation}
where the subscript $s$ indicates the value of the
considered quantity at the free surface. 

Assuming negligible the air viscosity, the continuity
of stresses at the free boundary imposes
\begin{equation}
\Sigma_T {\bf n}_s = -p^a {\bf n}_s,
\label{eq:BC_h}
\end{equation}
where $p^a=p^a(t,x)$ is a given function corresponding to the
atmospheric pressure.

\subsubsection{Bottom conditions}

The kinematic boundary condition at the bottom 
consists in a classical no-penetration condition:
\begin{equation}
\frac{\partial z_b}{\partial t}  + u_b \frac{\partial z_b}{\partial x} - w_b = 0,
\label{eq:bottom} 
\end{equation}
that reduces to ${\bf u}_b \cdot {\bf n}_b = 0$ when $z_b$ does not
depend on time $t$.

For the stresses at the bottom we consider a wall law under the form
\begin{equation*}
{\bf t}_b \cdot \Sigma_T {\bf n}_b = \kappa {\bf u}_b \cdot {\bf t}_b,
\label{eq:BC_z_b}
\end{equation*}
where ${\bf t}_b$ is a unit vector satisfying ${\bf t}_b \cdot {\bf n}_b = 0$.
If $\kappa({\bf u_b},h)$ is constant then
we recover a Navier friction condition as in \cite{gerbeau}. Introducing
a laminar friction $k_l$ and a turbulent friction $k_t$, we use the expression 
$$\kappa({\bf u_b},h) = k_l + k_t h |{\bf u_b}|,$$
corresponding to the boundary condition used in \cite{marche}. Another
form of $\kappa({\bf u_b},h)$ is used in \cite{bouchut}, and for other
wall laws the reader can also refer
to \cite{valentin}. Due to thermo-mechanical considerations, in the
sequel we will suppose $\kappa({\bf u_b},h) \geq 0$, and $\kappa({\bf
u_b},h)$ will be  often simply denoted by $\kappa$.

\subsection{Energy balance}
We recall the fondamental stability property related to the fact that
the hydrostatic Navier-Stokes system admits a mechanical energy
\begin{equation}
E(t,x,z) = \frac{u^2}{2} + gz,
\label{eq:energy_exp}
\end{equation}
leading to the following relation for smooth solutions
\begin{multline}
\frac{\partial}{\partial t}\int_{z_b}^{\eta} (E + p^a)\ dz +
\frac{\partial}{\partial x}\int_{z_b}^{\eta} \left[ u\bigl( E + p
  \bigr) -\mu \left( 2u\frac{\partial u}{\partial x} + w\left(
      \frac{\partial u}{\partial z} + \frac{\partial w}{\partial x}\right)\right)\right] dz
\\
= -2\mu\int \left[ \left(\frac{\partial u}{\partial x}\right)^2 +
  \frac{1}{2}\left(\frac{\partial u}{\partial z} + \frac{\partial w}{\partial x}\right)^2 + \left(\frac{\partial w}{\partial z}\right)^2 \right] dz + h\frac{\partial p^a}{\partial t} +
(\left. p\right|_b-p^a)\frac{\partial z_b}{\partial t} - \kappa u_b^2.
\label{eq:energy_eq}
\end{multline}
For the sake of simplicity, in the following we neglect the
variations of the atmospheric pressure $p^a$ i.e. $p^a =
p^a_0$ with $p^a_0=0$ and we also consider
\begin{equation}
\frac{\partial z_b}{\partial t} = 0.
\label{eq:zb}
\end{equation}
It follows the right hand side of \eqref{eq:energy_eq} is nonpositive and the mean energy 
$$
\bar{E}(t,x)=\int_{z_b}^{\eta} E(t,x,z) \, dz
$$
is not increasing in time.

\section{The layer-averaged Euler system}
\label{sec:av_euler}

In this section, we present a simplified derivation of the layer-averaged system that was introduced in \cite{JSM_M2AN}. As the main result of this work is related to a discrete entropy inequality, we focus on the energy property of the model.\\

Neglecting the viscous effects (we come back to the Navier-Stokes system in Section \ref{sec:lans}), we consider the hydrostatic Euler equations in a conservative form
\begin{eqnarray}
& & \frac{\partial u}{\partial x}+\frac{\partial w}{\partial z} =0,\label{eq:euler_2d1}\\
& & \frac{\partial u}{\partial t } + \frac{\partial u^2}{\partial x } +\frac{\partial uw }{\partial z }+ \frac{\partial p}{\partial x } = 0,\label{eq:euler_2d2}\\
& & 
\hspace*{3cm} \frac{\partial p}{\partial z} = -
g.\label{eq:euler_2d3}
\end{eqnarray}
Kinematic boundary conditions \eqref{eq:free_surf} and \eqref{eq:bottom} remain unchanged 
while the equality of stresses at the free surface \eqref{eq:BC_h} reduces to
\begin{equation}
p_s = 0.
\label{eq:ps}
\end{equation}
The energy balance~\eqref{eq:energy_eq} reduces to the equality
\begin{equation}
\frac{\partial}{\partial t}\int_{z_b}^{\eta} E\ dz +
\frac{\partial}{\partial x}\int_{z_b}^{\eta} u\bigl( E + p
  \bigr) dz = 0.
\label{eq:energy_euler_eq}
\end{equation}
Note that using \eqref{eq:euler_2d3} and \eqref{eq:ps} the pressure $p(t,x,z)$ can be computed has a function depending only of the free surface and the vertical coordinate. Moreover, using the divergence free condition \eqref{eq:euler_2d1} and the boundary condition \eqref{eq:bottom}, the vertical velocity can be computed as a function of the horizontal velocity and the vertical coordinate. It follows the unknowns of the system reduce to the water depth $h(t,x)=\eta(t,x)-z_b(x)$ and the horizontal velocity $u(t,x,z)$, that will appear even more explicitly in the layer averaged version presented below.

\subsection{Discretization of the fluid domain}

The interval $[z_b,\eta]$ is
divided into $N$ layers $\{L_\alpha\}_{\alpha\in\{1,\ldots,N\}}$ of
thickness $l_\alpha h(t,x)$ where each layer
$L_\alpha$ corresponds to the points satisfying $z \in L_\alpha(t,x) = [z_{\alpha-1/2},z_{\alpha+1/2}]$
with
\begin{equation}
\left\{\begin{array}{l}
z_{\alpha+1/2}(t,x) = z_b(x) + \sum_{j=1}^\alpha l_j h(t,x),\, z_{1/2}(x)=z_b(x)\\
h_\alpha(t,x) = z_{\alpha+1/2}(t,x) - z_{\alpha-1/2}(t,x)=l_\alpha h(t,x),
 \quad \alpha\in [1,\ldots,N],
\end{array}\right.
\label{eq:layer}
\end{equation}
 with $l_j>0, \quad \sum_{j=1}^N l_j=1$, see
Fig.~\ref{fig:notations_mc}.

\begin{figure}[hbtp]
\begin{center}
\caption{Notations for the multilayer approach.}
\label{fig:notations_mc}
\end{center}
\end{figure}

\subsection{Layer-averaged model}
\label{subsec:depth-averaging}

Let us consider the space
$\mathbb{P}_{0,h}^{N,t}$ of piecewise constant functions defined by
$$\mathbb{P}_{0,h}^{N,t} = \left\{ \1_{z \in L_\alpha(t,x)}(z),\quad \alpha\in\{1,\ldots,N\}\right\},$$
where $\1_{z \in L_\alpha(t,x)}(z)$ is the characteristic
function of the interval $L_\alpha(t,x)$. Using this formalism, the
projection of $u$, $w$ on $\mathbb{P}_{0,h}^{N,t}$ is a
piecewise constant function defined by
\begin{equation}
X^{N}(x,z,\{z_\alpha\},t)  = \sum_{\alpha=1}^N
\1_{[z_{\alpha-1/2},z_{\alpha+1/2}]}(z)X_\alpha(t,x),
\label{eq:ulayer}
\end{equation}
for $X \in (u,w)$.

In the following we focus on the following layer-averaged model
approximating the incompressible hydrostatic Euler system~\eqref{eq:euler_2d1}-\eqref{eq:euler_2d3}
\begin{eqnarray}
&& \frac{\partial h}{\partial t }   + 
\sum_{\alpha=1}^N \frac{\partial h_\alpha u_\alpha }{\partial x } =0.
\label{eq:HH}\\
&&\frac{\partial  h_\alpha u_\alpha}{\partial t } +
\frac{\partial }{\partial x }\left(h_\alpha u^2_\alpha +
  \frac{g}{2}h_\alpha h\right) = -gh_\alpha \frac{\partial
  z_b}{\partial x} +
u_{\alpha+1/2}G_{\alpha+1/2} -
u_{\alpha-1/2}G_{\alpha-1/2},\label{eq:eq44}
\end{eqnarray}
where the mass exchange terms $G_{\alpha+1/2}$ satisfy
\begin{eqnarray}
G_{\alpha+1/2} & = & \sum_{j=1}^\alpha \left(\frac{\partial h_j}{\partial t }   + 
\frac{\partial h_j u_j}{\partial x } \right),
\label{eq:Qalphabis}\\
G_{N+1/2} & = & G_{1/2} = 0,\label{eq:Qlim}
\end{eqnarray}
and the interface velocities $u_{\alpha+1/2}$ are defined using an upwinding strategy
\begin{equation}
u_{\alpha+1/2} =
\left\{\begin{array}{ll}
u_\alpha & \mbox{if } \;G_{\alpha+1/2} \leq 0\\
u_{\alpha+1} & \mbox{if } \;G_{\alpha+1/2} > 0.
\end{array}\right.
\label{eq:upwind_uT}
\end{equation}

\begin{remark}
In the monolayer case $\alpha=1$, system \eqref{eq:HH}-\eqref{eq:eq44} with boundary conditions \eqref{eq:Qlim} reduces to the classical shallow water equations, see \cite{bouchut_book,bristeau1}.
\end{remark}

Relation~\eqref{eq:eq44} is obtained by integrating the momentum equation \eqref{eq:euler_2d2}
on the layer $L_\alpha$ while the integration of the divergence equation \eqref{eq:euler_2d1} on the same domain leads to
\begin{equation}
\frac{\partial h_\alpha }{\partial t }   + \frac{\partial h_\alpha u_\alpha }{\partial x } =
G_{\alpha+1/2} - G_{\alpha-1/2}.
\label{eq:H}
\end{equation}
Relations \eqref{eq:Qlim} directly follow from the kinematic boundary conditions  \eqref{eq:free_surf} and \eqref{eq:bottom}. Then global mass equation \eqref{eq:HH} and definition of the mass exchange terms \eqref{eq:Qalphabis} are deduced by summing relations \eqref{eq:H} over layers $L_j$ for $j$ varying from $1$ to $N$ or from $1$ to $\alpha$ respectively.

\begin{remark}
The mass exchange terms $G_{\alpha+1/2}$ can be defined by using only derivatives in space that correspond to partial mass fluxes. It follows from the global mass equation \eqref{eq:HH} and the definition of the layer depth \eqref{eq:layer}
that relation \eqref{eq:Qalphabis} may be written
\begin{equation}
G_{\alpha+1/2}  =  \sum_{j=1}^\alpha \left( 
\frac{\partial h_j u_j}{\partial x } - l_j \sum_{k=1}^N \frac{\partial h_k u_k}{\partial x } \right).
\label{eq:Qalphabis2}
\end{equation}
\end{remark}

\begin{proposition}
The layer-averaged system~\eqref{eq:HH}-\eqref{eq:eq44} admits, for smooth solutions, the
layer energy balance
\begin{eqnarray}
& & \hspace*{-0.5cm}\frac{\partial }{\partial t} E_{\alpha} + \frac{\partial}{\partial x}
\left( u_\alpha\left(E_{\alpha} + \frac{gh}{2} h_\alpha \right)
\right) =
\left( u_{\alpha+1/2}u_{\alpha} -\frac{u_{\alpha}^2}{2} + g\eta\right) G_{\alpha+1/2}\nonumber \\
& & 
-\left( u_{\alpha-1/2}u_\alpha -\frac{u_{\alpha}^2}{2} +  g\eta\right) G_{\alpha-1/2},
\label{eq:energy_mcl}
\end{eqnarray}
where $E_\alpha$ is defined by 
$$E_{\alpha} = h_\alpha\left[\frac{ u_\alpha^2}{2} +g \left(\frac{h}{2}+z_b\right)\right].$$

It follows the global energy inequality 
\begin{eqnarray}
\frac{\partial }{\partial t} \left( \sum_{\alpha=1}^N E_{\alpha}\right) + \frac{\partial}{\partial x}
\left(\sum_{\alpha=1}^N u_\alpha\left(E_{\alpha} +
  \frac{gh}{2} h_\alpha \right) \right) = -\sum_{\alpha=1}^{N-1} \frac{1}{2}(u_{\alpha+1}
- u_\alpha)^2 |G_{\alpha+1/2}|.
\label{eq:energy_glol}
\end{eqnarray}
\label{prop:energy_bal}
where the right hand side is obviously nonpositive.
\end{proposition}

\begin{remark}
The layer energy $E_\alpha$ is not the meanvalue on a layer of the
pointwise energy $E$ associated to Euler equations and defined by
\eqref{eq:energy_exp}. The reason is we are not interested in a local kinetic energy associated to momentum equation per layer \eqref{eq:eq44} but in a global potential energy associated to the global mass conservation \eqref{eq:HH}. According to that, it is easy to check that 
$$
\sum_{\alpha=1}^N E_{\alpha} = \sum_{\alpha=1}^N \frac{h_\alpha
  u_\alpha^2}{2} + gh \left[\frac{h}{2} + z_b\right].
$$
Relation \eqref{eq:energy_glol} has then to be compared to relation \eqref{eq:energy_euler_eq} for  Euler equations. It appears that the vertical layer-averaging introduces numerical diffusion, as it is usually the case for spatial discretization associated to an upwinding strategy, see \eqref{eq:upwind_uT}.
\end{remark}

\begin{proof}[Proof of prop.~\ref{prop:energy_bal}]
Numerical computations to obtain relation \eqref{eq:energy_mcl} from the layer mass \eqref{eq:H} and momentum \eqref{eq:eq44} equations are a straightforward generalization of what is usually done for the classical shallow water model. More precisely, multiplying momentum equation \eqref{eq:eq44} by $u_\alpha$ and using \eqref{eq:H}  leads to
\begin{eqnarray*}
&& \frac{\partial }{\partial t } \left( \frac{h_\alpha u_\alpha^2}{2} \right) + \frac{\partial }{\partial x}
\left( u_\alpha \frac{h_\alpha u_\alpha^2}{2} \right) +
\frac{\partial }{\partial x }\left(\frac{gh}{2} h_\alpha\right)u_\alpha
+ g h_\alpha u_\alpha \frac{\partial z_b}{\partial x} \\
&& \qquad \qquad   = 
\left(u_{\alpha+1/2}u_\alpha - \frac{u_\alpha^2}{2}\right) G_{\alpha+1/2}
- \left(u_{\alpha-1/2}u_\alpha - \frac{u_\alpha^2}{2}\right) G_{\alpha-1/2},
\end{eqnarray*}
Now multiplying mass equation \eqref{eq:H} by $g\eta=g(h+z_b)$ and using relations \eqref{eq:zb} and \eqref{eq:layer} leads to
$$
\frac{\partial }{\partial t} (g z_b h_\alpha)+ \frac{\partial}{\partial t} \left(\frac{gh}{2} h_\alpha\right) + gz_b \frac{\partial h_\alpha u_\alpha }{\partial x } + gh \frac{\partial h_\alpha u_\alpha }{\partial x } =
\left(G_{\alpha+1/2} - G_{\alpha-1/2}\right) g\eta.
$$
Adding both relations and using \eqref{eq:layer} for the pressure term, we obtain the energy relation per layer \eqref{eq:energy_mcl}. The global energy inequality \eqref{prop:energy_bal} follows by adding these relations for all layers and using the upwind definition of the interface velocities \eqref{eq:upwind_uT}.
\end{proof}

\begin{remark}
Note that multiplying relation \eqref{eq:euler_2d1} and integrating over the layer $L_\alpha$ leads to the following equality
\begin{eqnarray*}
&& \frac{\partial}{\partial t} \left(\frac{z_{\alpha+1/2}^2 - z_{\alpha-1/2}^2}{2}\right) +
\frac{\partial}{\partial x} \left( \frac{z_{\alpha+1/2}^2 -
z_{\alpha-1/2}^2}{2}u_\alpha\right)  dz =  h_\alpha
w_\alpha\nonumber\\
&& \hspace*{6cm} + z_{\alpha+1/2}G_{\alpha+1/2} - z_{\alpha-1/2}
G_{\alpha-1/2},
\end{eqnarray*}
that may be used as a postprocessing to compute the layer vertical velocity $w_\alpha$.
\end{remark}

\subsection{Kinetic description}

In this paragraph we first give a kinetic interpretation of the
system~\eqref{eq:HH}-\eqref{eq:energy_mcl} and then we establish some properties of the
proposed discrete scheme. It is a generalization to the layer-averaged framework of kinetic interpretations proposed for other fluid models, see \cite{bristeau,simeoni,perthame}. Note that a first kinetic interpretation of the layer-averaged model was introduced in \cite{JSM_M2AN}. Here we propose an improved version and we derive an energy balance at the kinetic level.

\subsubsection{Kinetic interpretation}

Let us define the vector of unknowns
\begin{equation}
U_\alpha = (h_\alpha, h_\alpha u_\alpha)^T,\qquad U = (h, h_1 u_1,\ldots, h_N u_N)^T,
\label{eq:stateU}
\end{equation}
we also denote $q_{\alpha}= h_{\alpha} u_{\alpha}$. 

To build Gibbs equilibria, we choose the function
\begin{equation}
\chi_0 (z) = \frac{1}{\pi}\sqrt{1 - \frac{z^2}{4}}.
\label{eq:chi0}
\end{equation}
This choice corresponds to the classical kinetic maxwellian used in \cite{simeoni} for example and that is defined for $\xi \in \R$ by
\begin{equation}
	M_\alpha = M(U_\alpha,\xi)= \frac{h_\alpha}{c} \chi_0 \left(
          \frac{\xi - u_\alpha}{c}\right) =
        \frac{l_\alpha}{g\pi}\Bigl(2gh-(\xi-u_\alpha)^2\Bigr)_+^{1/2}
        = l_\alpha \overline{M}_\alpha,
	\label{eq:kinmaxw}
\end{equation}
with
$$c = \sqrt{\frac{g}{2}h}.$$
The definition of $\overline{M}_\alpha$ given by
Eq.~\eqref{eq:kinmaxw} will be used extensively in the following.

The kinetic maxwellian satisfies the following moment relations,
\begin{equation}
\begin{array}{c}
	\dsp \int_\R M(U_\alpha,\xi)\,d\xi=h_\alpha,\qquad \int_\R \xi
        M(U_\alpha,\xi)\,d\xi=h_\alpha u_\alpha,\\
	\dsp\int_\R \xi^2 M(U_\alpha,\xi)\,d\xi=h_\alpha u_\alpha^2+h_\alpha \frac{gh}{2}.
	\label{eq:kinmomM}
	\end{array}
\end{equation}
Now we introduce a second list of Gibbs equilibria $N_{\alpha+1/2}$ associated to the mass exchange terms between layers and defined by 
\begin{equation}
N_{\alpha+1/2} = N(U_\alpha,U_{\alpha+1},\xi) = \frac{G_{\alpha+1/2}}{c} \ \chi_0 \left(\frac{\xi -
  u_{\alpha+1/2}}{c}\right), \quad \alpha=1,\ldots,N-1
\label{eq:Nbis}
\end{equation}
completed by the boundary conditions $N_{1/2} = N_{N+1/2} = 0$, see \eqref{eq:Qlim}.
Due to the upwind definition of the interface velocity \eqref{eq:upwind_uT}, definition \eqref{eq:Nbis} is equivalent to
\begin{equation}
N_{\alpha+1/2} = \frac{G_{\alpha+1/2}}{h} \overline{M}_{\alpha+1/2}, \qquad 
\overline{M}_{\alpha+1/2} = 
\left\{\begin{array}{ll}
\overline{M}_\alpha & \mbox{if } G_{\alpha+1/2} \leq 0\\
\overline{M}_{\alpha+1} & \mbox{if } G_{\alpha+1/2} \geq 0
\end{array}\right.
\label{eq:defN_M}
\end{equation}
It follows from relations \eqref{eq:kinmomM} that 
\begin{equation}
\int_\R N_{\alpha+1/2} \, d\xi = {G_{\alpha+1/2}}, 
\quad \int_\R \xi N_{\alpha+1/2} \,d\xi = {G_{\alpha+1/2}} u_{\alpha+1/2}.
\label{eq:kinmomN}
\end{equation}

\begin{remark}
It is clear from definitions \eqref{eq:Nbis} and \eqref{eq:chi0} that $N_{\alpha+1/2}$ is well-defined for all values of $h > 0$. It is not easy to prove it remains bounded when the water depth vanishes. But it is sufficient for our purpose to note that one can characterize its behavior since, in the sense of distributions,
$$
\frac{1}{c}\chi_0\left(\frac{\xi-u}{c}\right) \underset{h\rightarrow 0}{\longrightarrow} \delta_u(\xi)
$$
\label{remark:dirac}
\end{remark}

%

We are now equipped to exhibit the kinetic interpretation of the layer-averaged model (\ref{eq:HH})-(\ref{eq:eq44}).
\begin{proposition}
The functions $(h,u^N)$ defined by~\eqref{eq:ulayer} are strong solutions of the system 
(\ref{eq:HH})-(\ref{eq:eq44}) if and only if the sets of
equilibria $\{M_\alpha\}_{\alpha=1}^N$,~$\{N_{\alpha+1/2}\}_{\alpha=0}^N$
are solutions of the kinetic equations defined by
\begin{equation}
({\cal B}_\alpha)\qquad\frac{\partial M_\alpha}{\partial t} + \xi \frac{\partial M_\alpha}{\partial x} 
  -g \frac{\partial z_b}{\partial x} \frac{\partial M_\alpha}{\partial \xi} - N_{\alpha+1/2} + N_{\alpha-1/2} 
 = Q_{\alpha}, \qquad \alpha=1,\ldots,N
\label{eq:gibbsbis}
\end{equation}
The quantities $Q_{\alpha} = Q_{\alpha}(t,x,\xi)$ 
 are ``collision terms''  equal to zero at the
macroscopic level, i.e.\ they satisfy  a.e.\ for values of $(t,x)$
\begin{equation}
\int_{\mathbb{R}} Q_{\alpha} d\xi = 0,\quad \int_{\mathbb{R}} \xi
Q_{\alpha} d\xi =0.
\label{eq:collisionbis}
\end{equation}
\label{prop:kinetic_sv_mcl}
\end{proposition}


\begin{proof}[Proof of proposition~\ref{prop:kinetic_sv_mcl}]
The proof relies on averages w.r.t the variable $\xi$ of
Eq.~\eqref{eq:gibbsbis} by using relations \eqref{eq:kinmomM} and
\eqref{eq:kinmomN}. Then using~\eqref{eq:collisionbis}, the quantities
$$\sum_1^N \int_\R ({\cal B}_\alpha)\ d\xi,\quad\mbox{and}\quad \int_\R \xi ({\cal B}_\alpha)\ d\xi,$$
respectively give~\eqref{eq:HH} and \eqref{eq:eq44} that completes
the proof.
\end{proof}

There are a lot of functions that satisfy the integral relations
\eqref{eq:kinmomM}. The interest of the kinetic maxwellian defined by \eqref{eq:kinmaxw} lies in its link with the kinetic entropy
\begin{equation}
	H_\alpha(f,\xi,z_b)= l_\alpha \left(
          \frac{\xi^2}{2}f+\frac{g^2\pi^2}{6}f^3+gz_b f \right) =
        l_\alpha H(f,\xi,z_b),
	\label{eq:kinH}
\end{equation}
where $f\geq 0$, $\xi\in\R$, $z_b\in\R$. Indeed one can check the relations
\begin{equation}\begin{array}{c}
	\dsp
E_\alpha  
= l_\alpha \int_\R H(\overline{M}_\alpha,\xi,z_b) d\xi,\\
\dsp u_\alpha
\left(E_\alpha + \frac{g}{2} h_\alpha h\right) = l_\alpha \int_\R \xi H(\overline{M}_\alpha,\xi,z_b) d\xi.
\label{eq:kinmomE}
	\end{array}
\end{equation}


\begin{proposition}
The solutions of the kinetic equation~\eqref{eq:gibbsbis} are entropy solutions in the sense they
satisfy on each layer the kinetic energy inequality 
\begin{align}
\dsp \frac{\partial}{\partial t}H(\overline{M}_\alpha,\xi,z_b) + &\xi \frac{\partial}{\partial x}H(\overline{M}_\alpha,\xi,z_b)
- g \frac{\partial z_b}{\partial x}\frac{\partial}{\partial
  \xi}H(\overline{M}_\alpha,\xi,z_b) \nonumber\\
\leq &
 \frac{G_{\alpha+1/2}}{h}H(\overline{M}_{\alpha+1/2},\xi,z_b)
                  -\frac{G_{\alpha-1/2}}{h}H(\overline{M}_{\alpha-1/2},\xi,z_b) \nonumber \\
& +\frac{g^2\pi^2}{6}\frac{G_{\alpha+1/2}}{h}
(\overline{M}_{\alpha+1/2}+2\overline{M}_{\alpha})(\overline{M}_{\alpha+1/2}-\overline{M}_{\alpha})^2\nonumber\\
&-\frac{g^2\pi^2}{6}\frac{G_{\alpha-1/2}}{h}
(\overline{M}_{\alpha-1/2}+2\overline{M}_{\alpha})(\overline{M}_{\alpha-1/2}-\overline{M}_{\alpha})^2\nonumber\\&
+\frac{g^2\pi^2}{3}\frac{G_{\alpha+1/2} - G_\alpha}{h}\overline
M_\alpha^3
+\partial_1 H (\overline{M}_\alpha,\xi,z_b)Q_\alpha
\label{eq:kin_entro1}
\end{align}
Integration in $\xi$ and sum on $\alpha$ of relations~\eqref{eq:kin_entro1} lead to a global macroscopic energy inequality, that is analog to~\eqref{eq:energy_glol}
\begin{align}
&\frac{\partial }{\partial t} \left( \sum_{\alpha=1}^N E_{\alpha}\right) + \frac{\partial}{\partial x}
\left(\sum_{\alpha=1}^N u_\alpha\left(E_{\alpha} +
  \frac{gh}{2} h_\alpha \right) \right) \nonumber \\
& \qquad \leq -\frac{g^2\pi^2}{6}\sum_{\alpha=1}^{N-1}\frac{|G_{\alpha+1/2}|}{h}
\int_\R (\overline{M}_{\alpha+1}+2\overline{M}_{\alpha})(\overline{M}_{\alpha+1}-\overline{M}_{\alpha})^2 d\xi.
\label{eq:energ_kin}
\end{align}
\label{prop:entropy_kin}
\end{proposition}

\begin{proof}
[Proof of Proposition~\ref{prop:entropy_kin}]
Kinetic energy inequality \eqref{eq:kin_entro1} is obtained by multiplying the kinetic equation~\eqref{eq:gibbsbis} by
$\partial_1 H(\overline{M}_\alpha,\xi,z_b)$, where $\partial_i$
denotes the derivative in the first $i^{th}$ variable.
Indeed, it is easy to see that (remember the topography $z_b$ does not depend on time)
$$\partial_1  H(\overline{M}_\alpha,\xi,z_b)
\frac{\partial M_\alpha}{\partial t} = 
\frac{\partial}{\partial t} H(\overline{M}_\alpha,\xi,z_b),
$$
likewise we have
$$\xi \partial_1 H(\overline{M}_\alpha,\xi,z_b)
\frac{\partial M_\alpha}{\partial x} =   \xi
\frac{\partial}{\partial x} 
  H(\overline{M}_\alpha,\xi,z_b) - \xi \partial_3 H(\overline{M}_\alpha,\xi,z_b) \frac{\partial z_b}{\partial x},
$$
and
$$- g \frac{\partial z_b}{\partial x} 
\partial_1 H(\overline{M}_\alpha,\xi,z_b)
\frac{\partial M_\alpha}{\partial \xi} = - 
g \frac{\partial z_b}{\partial x}
\frac{\partial}{\partial \xi}
H(\overline{M}_\alpha,\xi,z_b)
+ g \frac{\partial z_b}{\partial x}
\partial_2 H(\overline{M}_\alpha,\xi,z_b).
$$
But it follows from definition \eqref{eq:kinH} of the kinetic energy that
$$
g \frac{\partial z_b}{\partial x}
\partial_2 H(\overline{M}_\alpha,\xi,z_b)
- \xi \partial_3 H(\overline{M}_\alpha,\xi,z_b) \frac{\partial z_b}{\partial x}=0.
$$
It remains to obtain a suitable expression for the quantity $\partial_1 H(\overline{M}_\alpha,\xi,z_b) N_{\alpha+1/2}$. 
Let's denote
$$N_\alpha = \frac{1}{2}(N_{\alpha+1/2} + N_{\alpha-1/2}), \qquad G_\alpha = \frac{1}{2}(G_{\alpha+1/2} + G_{\alpha-1/2}),$$
it follows from \eqref{eq:defN_M} that
\begin{multline}
\partial_1 H(\overline{M}_\alpha,\xi,z_b) ( N_{\alpha+1/2} -
N_\alpha) = \frac{G_{\alpha+1/2}}{h} \partial_1 H(\overline{M}_\alpha,\xi,z_b) ( \overline M_{\alpha+1/2} -
\overline M_\alpha) \\
+ \frac{G_{\alpha+1/2} - G_\alpha}{h}
\partial_1 H(\overline{M}_\alpha,\xi,z_b) \overline M_\alpha,
\label{eq:N1}
\end{multline}
and the definition \eqref{eq:kinH} of $H_\alpha(\overline{M}_\alpha,\xi,z_b)$ gives us
\begin{multline}
H(\overline{M}_{\alpha+1/2},\xi,z_b) =
H(\overline{M}_{\alpha},\xi,z_b) + 
\partial_1 H(\overline{M}_\alpha,\xi,z_b) ( \overline M_{\alpha+1/2} -
\overline M_\alpha) \\
+\frac{g^2\pi^2}{6}
(\overline{M}_{\alpha+1/2}+2\overline{M}_{\alpha})(\overline{M}_{\alpha+1/2}-\overline{M}_{\alpha})^2.
\label{eq:N2}
\end{multline}
Relation~\eqref{eq:N1} together with~\eqref{eq:N2} leads to
\begin{multline*}
\partial_1 H(\overline{M}_\alpha,\xi,z_b) ( N_{\alpha+1/2} -
N_\alpha) = \frac{G_{\alpha+1/2}}{h}
H(\overline{M}_{\alpha+1/2},\xi,z_b) - \frac{G_{\alpha}}{h} H(\overline{M}_{\alpha},\xi,z_b)  \\
-\frac{g^2\pi^2}{6}\frac{G_{\alpha+1/2}}{h}
(\overline{M}_{\alpha+1/2}+2\overline{M}_{\alpha})(\overline{M}_{\alpha+1/2}-\overline{M}_{\alpha})^2
+ \frac{g^2\pi^2}{3}\frac{G_{\alpha+1/2} - G_\alpha}{h}\overline M_\alpha^3,
\label{eq:N3}
\end{multline*}
and the same kind of relation occurs for $\partial_1 H_\alpha(\overline{M}_\alpha,\xi,z_b) ( N_{\alpha-1/2} - N_\alpha)$. 
Kinetic energy inequality per layer \eqref{eq:kin_entro1} obviously follows.\\

In order to derive the global energy inequality \eqref{eq:energ_kin}, let's now detail the right hand side of relation \eqref{eq:kin_entro1}, the left hand side being treated using integral relations \eqref{eq:kinmomE}. The first line involves vertical kinetic exchange terms that vanish when summing on the layers. Second and third lines involve nonpositive terms due to the upwind definition \eqref{eq:defN_M} of the interface Maxwellian $\overline{M}_{\alpha+1/2}$. They will lead to the nonpositive right hand side in relation \eqref{eq:energ_kin}. Finally, the terms in the last line have no particular sign at the kinetic level. But after integration in $\xi$, one observes that 
$$
\int_\R \frac{g^2\pi^2}{3}
\frac{(G_{\alpha+1/2}-G_{\alpha-1/2})}{h}\overline{M}^3_{\alpha} d\xi
= 
\frac{g}{2}h (G_{\alpha+1/2}-G_{\alpha-1/2})
$$
that can be interpreted as a macroscopic vertical exchange term and then vanishes when summing on the layers. Moreover, due to the particular choice \eqref{eq:kinmaxw} for the kinetic maxwellian $\overline{M}_{\alpha}$, one has
\begin{equation}
\partial_1 H(\overline{M}_\alpha,\xi,z_b) = \frac{\xi^2}{2}+\frac{g^2\pi^2}{2}\overline{M}_\alpha^2+gz_b=\frac{\xi^2}{2}+g\eta-\frac{(\xi-u_\alpha)^2}{2}=-\frac{u_\alpha^2}{2}+g\eta+u_\alpha \xi
\label{eq:d1H}
\end{equation}
and it hence follows from integral relations \eqref{eq:collisionbis} on the collision term that the last term on the right hand side of \eqref{eq:kin_entro1} vanishes when integrating in $\xi$.
\end{proof}

We end this section with a last result that extends to the present layer-averaged framework a  subdifferential inequality and an energy minimization principle that were exhibited in the classical shallow water framework in \cite{JSM_entro}, see also \cite{BGK} for the first use of this approach. This result will be used in the next section to extend the entropy inequality to the fully discrete case.\\


\begin{lemma}
\label{lemma:energy}
\noindent (i) For any $h_\alpha\geq 0$, $u_\alpha\in\R$, $f\geq 0$,
$\xi\in\R$ and with the definition of $E_\alpha$ given by~\eqref{eq:kinmomE}
\begin{equation}
	H(f,\xi,z_b)\geq H\bigl(\overline{M}_\alpha,\xi,z_b\bigr)+E_\alpha^\prime(U_\alpha)\kxi
	\bigl(f-\overline{M}_\alpha\bigr).
	\label{eq:subdiff}
\end{equation}

\noindent (ii) For any function $f(\xi)$ nonnegative, setting $h_\alpha=\int_\R
f(\xi)d\xi$ and $h_\alpha u_\alpha=\int_\R \xi f(\xi)d\xi$ (assumed finite), one has
\begin{equation*}
	E_\alpha 
	\leq l_\alpha \int_\R H\bigl(f(\xi),\xi,z_b\bigr)\,d\xi.
	\label{eq:entropmini}
\end{equation*}
\end{lemma}
\begin{proof}
The property (ii) follows from (i) by taking $f=f(\xi)$ and integrating \eqref{eq:subdiff}
with respect to $\xi$ since $f$ and $\overline{M}_\alpha$ share the same first two moments. For proving (i), we first notice that (remember that $h_\alpha=l_\alpha h$)
\begin{equation*}
	E_\alpha^\prime(U_\alpha)=\bigl(gh + g z_b -u_\alpha^2/2,u_\alpha\bigr),
	\label{eq:etaprime}
\end{equation*}
and then 
\begin{equation*}
	E_\alpha^\prime(U_\alpha)\kxi
	=gh +g z_b-u_\alpha^2/2+\xi u_\alpha=\frac{\xi^2}{2}+g h +g z_b-\frac{(\xi-u_\alpha)^2}{2}.
	\label{eq:etaprime1xi}
\end{equation*}
Now the definition \eqref{eq:kinmaxw} of the maxwellian $\overline{M}_\alpha$ yields
\begin{equation*}
	gh - \frac{(\xi-u_\alpha)^2}{2}=\left\{\begin{array}{l}
	\dsp \frac{g^2\pi^2}{2}\overline{M}_\alpha^2\quad\mbox{if } \overline{M}_\alpha>0,\\
	\dsp\mbox{is nonpositive\quad if } \overline{M}_\alpha=0,
	\end{array}\right.
	\label{eq:idM}
\end{equation*}
Using relation \eqref{eq:d1H}, it follows that
\begin{equation}
	\partial_1 H\bigl(\overline{M}_\alpha,\xi,z_b\bigr)=\left\{\begin{array}{l}
	\dsp E^\prime_\alpha\kxi\quad\mbox{if }\overline{M}_\alpha>0,\\
	\dsp\geq E^\prime_\alpha \kxi \quad\mbox{if } \overline{M}_\alpha=0.
	\end{array}\right.
	\label{eq:idH'}
\end{equation}
We conclude using the convexity of $H_\alpha$ with respect to $f$, see definition \eqref{eq:kinH},
\begin{equation}\begin{array}{l}
	\dsp H(f,\xi,z_b)\geq
        H\bigl(\overline{M}_\alpha,\xi,z_b\bigr)+\partial_1 H\bigl(\overline{M}_\alpha,\xi,z_b\bigr)\bigl(f-\overline{M}_\alpha\bigr)\\
	\dsp \hphantom{H(f,\xi,z_b)}
	\geq H\bigl(\overline{M}_\alpha,\xi,z_b\bigr)+E^\prime_\alpha\kxi
	\bigl(f-\overline{M}_\alpha\bigr),
	\label{eq:convineqH_0}
	\end{array}
\end{equation}
which proves the claim.
\end{proof}

\subsubsection{Discrete model}

The method proposed in \cite{JSM_entro} in order to solve the Saint-Venant
system from its kinetic interpretation can be extended to the
system~\eqref{eq:HH}-\eqref{eq:eq44} and its kinetic interpretation given
in Proposition~\ref{prop:kinetic_sv_mcl}. It is the purpose of this
paragraph.

We would like to approximate the solution $U(t,x)$, see \eqref{eq:stateU}, $x \in \R$, $t
\geq 0$ of the system~\eqref{eq:HH}-\eqref{eq:eq44} by discrete values $U_i^n$, $i \in
\mathbb{Z}$, $ n \in \mathbb{N}$. In order to do so, we consider a grid of points $x_{i+1/2}$, $i \in
\mathbb{Z}$,
$$\ldots < x_{i-1/2} < x_{i+1/2} <  x_{i+3/2} < \ldots,$$
and we define the cells (or finite volumes) and their lengths
$$C_i = ]x_{i-1/2},x_{i+1/2}[,\qquad \Delta x_i = x_{i+1/2} - x_{i-1/2}.$$
We consider discrete times $t^n$ with $t^{n+1}=t^n+\Delta t^n$, and
we define the piecewise constant functions $U^n(x)$ corresponding to time $t^n$ and $z(x)$ as
\begin{equation*}
	U^n(x)= U^n_i= (h_i^n, q^n_{1,i}\ldots, q^n_{N,i})^T,\quad z_b(x)= z_{b,i},\quad\mbox{ for }x_{i-1/2}<x<x_{i+1/2}.
	\label{eq:U^npc}
\end{equation*}

A finite volume scheme for solving~\eqref{eq:HH}-\eqref{eq:eq44} is a formula of the form
\begin{equation}
	U^{n+1}_i=U^n_i-\sigma_i(F_{i+1/2-}-F_{i-1/2+}) + \Delta t^n S_i,
	\label{eq:upU0}
\end{equation}
where $\sigma_i=\Delta t^n/\Delta x_i$, telling how to compute the
values $U^{n+1}_i$ knowing $U_i^n$ and discretized values $z_{b,i}$ of
the topography. The quantity $S_i$ is a source term accounting for
the discrete momentum exchange terms between each layer
in~\eqref{eq:eq44}. Here we consider first-order explicit three points schemes where
\begin{equation*}
	F_{i+1/2-} = \F_l(U_i^n,U_{i+1}^n,z_{i+1}-z_i),\qquad F_{i+1/2+} = \F_r(U_i^n,U_{i+1}^n,z_{i+1}-z_i).
\label{eq:flux_def}
\end{equation*}
The functions $\F_{l/r}(U_l,U_r,\Delta z)\in \R^2$ are the numerical fluxes, see \cite{bouchut_book}.\\

The proposed discrete scheme is based on the equivalence between the kinetic and the macroscopic levels stated in Proposition \ref{prop:kinetic_sv_mcl} and can be divided into three steps
\begin{itemize}
\item To construct the discrete kinetic maxwellian $M_{\alpha,i}^n$
  starting from the macroscopic quantities $U_i^n$ and the definition \eqref{eq:kinmaxw}
\begin{equation}	
	M^{n}_{\alpha,i}(\xi) = \frac{l_\alpha}{g\pi}\left(2gh^{n}_{\alpha,i}-(\xi-u^{n}_{\alpha,i})\right)^{1/2}_+
	\label{eq:Mnalphai}
\end{equation}
\item To update the kinetic quantities through a finite volume scheme that will be precised hereafter to compute the quantities $M_{\alpha,i}^{n+1,-}$. This step is in general performed without considering the collision term and it follows that $M_{\alpha,i}^{n+1,-}$ is no more of a maxwellian.
\item To compute the new macroscopic quantities $U_i^{n+1}$ as the integral of the kinetic quantities $M_{\alpha,i}^{n+1,-}$
\begin{equation*}
	U^{n+1}_{\alpha,i}=\int_\R \kxi M^{n+1-}_{\alpha,i}(\xi)\,d\xi.
	\label{eq:Un+1-i}
\end{equation*}
The difference between $M_{\alpha,i}^{n+1,-}$ and $M_{\alpha,i}^{n+1}$, computed using \eqref{eq:Mnalphai} at time $t^{n+1}$, can be seen as an instantaneous relaxation on the maxwellian, see \cite{bouchut_book}.
\end{itemize}
Such a kind of kinetic scheme was presented for the classical shallow water system in \cite{bristeau}, see also \cite{bouchut_book}, and for the layer-averaged system \eqref{eq:HH}-\eqref{eq:eq44} in \cite{JSM_M2AN}. In these works, the kinetic step was fully explicit. Here, and in order to demonstrate stability properties, we propose an implicit-explicit variant that is presented  in details in the next section. 
Note that for practical computations, the integration processes are not performed on the cell unknowns but directly to compute macroscopic fluxes, as it will be explained hereafter. It follows the presented scheme can be entirely written at the macroscopic level, avoiding expensive computations at the kinetic level. Nevertheless, the kinetic interpretation is an efficient way to demonstrate the properties of the scheme.

\subsubsection{Discrete kinetic equation}

Let us now detail the kinetic scheme we propose. It can be written in a one-step version
\begin{multline}
M_{\alpha,i}^{n+1-}=M_{\alpha,i}-\sigma_i\Bigl( \xi M_{\alpha,i+1/2} +
        \delta M_{\alpha,i+1/2-} - \xi M_{\alpha,i-1/2} - \delta M_{\alpha,i-1/2+}
	\Bigr) \\
+  \Delta t^n ( N_{\alpha+1/2,i}^{n+1-} - N_{\alpha-1/2,i}^{n+1-})
\label{eq:kinsc}
\end{multline}
or divided into an explicit and an implicit steps
\begin{align}
&M_{\alpha,i}^{n*}=M_{\alpha,i}-\sigma_i\Bigl( \xi M_{\alpha,i+1/2} +
        \delta M_{\alpha,i+1/2-} - \xi M_{\alpha,i-1/2} - \delta M_{\alpha,i-1/2+}
	\Bigr) \label{eq:kinsc1}\\
&M_{\alpha,i}^{n+1-} = M_{\alpha,i}^{n*}    +    \Delta t^n ( N_{\alpha+1/2,i}^{n+1-} - N_{\alpha-1/2,i}^{n+1-})
\label{eq:kinsc2}
\end{align}
with $\sigma_i=\Delta t^n/\Delta x_i$. To simplify the notations, we
omit the variable $\xi$ and the superscript $^n$. The quantities $M_{\alpha,i\pm
  1/2}$ and $\delta M_{\alpha,i+1/2\pm}$ respectively account for the
conservative part and the topography source term, their definitions
will be precised later.
The definition of the quantity $N_{\alpha+1/2,i}^{n+1-}$ 
requires a discrete extension of relation~\eqref{eq:defN_M}
\begin{equation}
N_{\alpha+1/2,i}^{n+1-} = \frac{G_{\alpha+1/2,i}}{h_i^{n*}} \overline{M}_{\alpha+1/2,i}^{n+1-} , \qquad 
\overline{M}_{\alpha+1/2_i}^{n+1-} = 
\left\{\begin{array}{ll}
\overline{M}_{\alpha,i}^{n+1-}& \mbox{if } G_{\alpha+1/2,i}^{n*} \leq 0\\
\overline{M}_{\alpha+1}^{n+1-} & \mbox{if } G_{\alpha+1/2,i}^{n*} \geq 0
\end{array}\right.
\label{eq:defN_Mbis}
\end{equation}
where the discrete mass exchange term $G_{\alpha+1/2,i}$
is computed using a discrete version of relation \eqref{eq:Qalphabis2}
\begin{multline}
\Delta x_i G_{\alpha+1/2,i} = \sum_{j=1}^\alpha \left(  \int_\R \xi
  (M_{j,i+1/2} - M_{j,i-1/2}) d\xi \right.\\
 \left. - l_j \sum_{p=1}^N  \int_\R \xi (M_{p,i+1/2} - M_{p,i-1/2}) d\xi
     \right).
\label{eq:Gdis3}
\end{multline}

\begin{remark}
The analysis of the behaviour of the quantity $N_{\alpha+1/2}$ when the water height vanishes is much more easy at the discrete level than at the continuous one, see remark~\ref{remark:dirac}.
Indeed, thanks to the choice of an implicit time discretization, we can prove that this quantity remains bounded if the time step does not vanish, that will be proved in the next section, see Th. \ref{thm:entropy_st}.
To prove the result, let us first note that the total water depth is not affected by the implicit step \eqref{eq:kinsc2} that takes into account the vertical exchange terms. It follows that $h_i^{n+1}=h_i^{n*}$ but also that it is not the case for the quantities $h_{\alpha,i}^{n+1}$ and $h_{\alpha,i}^{n*}$.
Now, considering Eq.~\eqref{eq:kinsc2} for the lowest layer and using the fact that $N_{1/2,i}^{n+1-}=0$ by definition, we get
$$\Delta t^n N_{3/2,i}^{n+1-} = M_{1,i}^{n+1-} - M_{1,i}^{n*},$$
and hence after integration in $\xi$ it comes
$$\Delta t^n l_1 G_{3/2,i} = h_{1,i}^{n+1} - h_{1,i}^{n*} = l_1 h_{i}^{n+1} - h_{1,i}^{n*}.$$
Since $\sum_{j=1}^N h_{j,i}^{n*} = h_{i}^{n*} = h_{i}^{n+1}$, we have $h_{1,i}^{n*} \leq h_{i}^{n+1}$ and this gives us the estimate
$$1 -\frac{1}{l_1} \leq \Delta t^n  \frac{G_{3/2,i}}{h_{i}^{n+1}} \leq 1.$$
Using the same process for each layer from the bottom to the top, one can prove the quantity 
${G_{\alpha+1/2,i}}/{h_{i}^{n+1}}$ is bounded for any $\alpha$, even when the water depth vanishes.
\label{rem:G_bound_dis}
\end{remark}

The explicit step \eqref{eq:kinsc1} is very similar to the kinetic scheme proposed in \cite{bristeau} and analysed in \cite{JSM_entro} for the classical shallow water problem.
We first prove hereafter that the implicit step \eqref{eq:kinsc2} leads to a well posed problem.
Then, in the next section, we prove the stability properties of the whole scheme \eqref{eq:kinsc}.
Using \eqref{eq:defN_Mbis} and \eqref{eq:Gdis3}, the implicit step \eqref{eq:kinsc2} can be written
\begin{align}
&- \Delta
t^n\frac{|G_{\alpha+1/2,i}|_+}{h_{\alpha+1,i}^{n+1}}
M_{\alpha+1,i}^{n+1-} + \left( 1 - \Delta t^n\frac{|G_{\alpha+1/2,i}|_- -
    |G_{\alpha-1/2,i}|_+}{h_{\alpha,i}^{n+1}}\right)
M_{\alpha,i}^{n+1-} \nonumber\\
&\qquad \qquad \qquad \qquad \qquad \qquad \qquad \qquad \qquad \qquad + \Delta
t^n\frac{|G_{\alpha-1/2,i}|_-}{h_{\alpha-1,i}^{n+1}}
M_{\alpha-1,i}^{n+1-} = M_{\alpha,i}^{n*} 
\label{eq:kin_st1}
\end{align}
that is equivalent to solve the linear system
\begin{equation}
\left({\bf I}_N+\Delta t G_{N,i}\right) M_{i}^{n+1-} = M_{i}^{n*}, \qquad M^k_i = \left(M^k_{1,i},...,M^k_{N,i}\right)^T
\label{eq:kinls}
\end{equation}
with
$$G_{N,i}=\begin{pmatrix}
-\frac{|G_{3/2,i}|_-}{h_{1,i}^{n+1}} &
-\frac{|G_{3/2,i}|_+}{h_{1,i}^{n+1}} & 0 & 0 & \cdots  & 0\\
\frac{|G_{3/2,i}|_-}{h_{2,i}^{n+1}} & \ddots & \ddots & 0 & \cdots &
0\\
0 & \ddots & \ddots & \ddots & 0 & 0\\
\vdots  & 0 & \frac{|G_{\alpha-1/2,i}|_-}{h_{\alpha,i}^{n+1}} &
-\frac{|G_{\alpha+1/2,i}|_- -
    |G_{\alpha-1/2,i}|_+}{h_{\alpha,i}^{n+1}}
& -\frac{|G_{\alpha+1/2,i}|_+}{h_{\alpha,i}^{n+1}} & 0 \\
\vdots & \ddots & 0 & \ddots & \ddots & -\frac{|G_{N-1/2,i}|_+}{h_{N-1,i}^{n+1}}\\
0 & \cdots & 0 & 0 & \frac{
    |G_{N-1/2,i}|_-}{h_{N,i}^{n+1}} & \frac{
    |G_{N-1/2,i}|_+}{h_{N,i}^{n+1}}
\end{pmatrix}.$$
\begin{lemma}
The matrix ${\bf A}_{N,i}={\bf I}_N + \Delta t {\bf G}_{N,i}$ satisfies the following properties
\begin{itemize}
\item[{\it (i)}] The matrix ${\bf A}_{N,i}$ is invertible for any $h_i^{n+1}>0$ and then the linear system \eqref{eq:kinls} has a unique solution.
\item[{\it (ii)}] Its inverse ${\bf A}_{N,i}^{-1}$ 
has only positive coefficients and then the kinetic density $M_{i}^{n+1-}$ is positive if $M_{i}^{n*}$ is.
\item[{\it (iii)}] For any vector $T$ with non negative entries
  i.e. $T_\alpha \geq 0$, for $1\leq \alpha\leq N$, one has
$$\| {\bf A}_{N,i}^{-T} T \|_\infty \leq \|T\|_\infty.$$
and then the solution of the linear system \eqref{eq:kinls} does not raise difficulties for any $h_i>0$ even if $h_i$ is arbitrarily small.
\end{itemize}
\label{lem:inverse}
\end{lemma}

\begin{proof}[Proof of lemma~\ref{lem:inverse}]
\begin{itemize}
\item[{\it (i)}]  Let us first note that if $h_i^{n+1}=0$, we do not solve the linear system \eqref{eq:kinls} but simply impose that $M_i^{n+1-}=0$. Now for any $h_i^{n+1}>0$, the matrix ${\bf A}^T_{N,i}$
is a strictly dominant diagonal matrix. It follows that ${\bf A}_{N,i}$ is invertible.
\item[{\it (ii)}] Denoting ${\bf G}_{N,i}^d$ (resp. ${\bf G}_{N,i}^{nd}$) the diagonal
(resp. non diagonal) part of ${\bf G}_{N,i}$ we can write
$${\bf A}_{N,i} = {\bf I}_N + \Delta t {\bf G}_{N,i} = ({\bf I}_N + \Delta t {\bf
  G}_{N,i}^d) \left( {\bf I}_N - ({\bf I}_N + \Delta t {\bf
    G}_{N,i}^d)^{-1} (-\Delta t G_{N,i}^{nd})\right),$$
where all the entries of the matrix
$${\bf J}_{N,i} = ({\bf I}_N + \Delta t {\bf
    G}_{N,i}^d)^{-1} (-\Delta t {\bf G}_{N,i}^{nd}),$$
are non negative and less than 1. And hence, we can write
$$({\bf I}_N + \Delta t {\bf G}_{N,i})^{-1} = \sum_{k=0}^\infty J_{N,i}^k,$$
proving all the entries of $({\bf I}_N + \Delta t {\bf G}_{N,i})^{-1}$
are non negative.
\item[{\it (iii)}] Let us consider the vector ${\bf 1}$ whose entries are all
equal to 1. Since we have
$$({\bf I}_N + \Delta t {\bf G}_{N,i})^T {\bf 1} = {\bf 1},$$
we also have
$${\bf 1} = ({\bf I}_N + \Delta t {\bf G}_{N,i})^{-T} {\bf 1}.$$
Now let $T$ be a vector whose entries $\{T_\alpha\}_{1 \leq
  \alpha \leq N}$ are non negative, then
$$({\bf I}_N + \Delta t {\bf G}_{N,i})^{-T} {\bf T} \leq ({\bf I}_N +
\Delta t {\bf G}_{N,i})^{-T} {\bf 1} \|{\bf T}\|_\infty = {\bf 1}
\|{\bf T}\|_\infty,$$
that completes the proof.
\end{itemize}
\end{proof}

\section{Properties of the scheme}
\label{sec:properties}

In this section, we examine the properties of the scheme~\eqref{eq:kinsc}.

\subsection{Without topography}

We first consider the problem without topography. 
The scheme~\eqref{eq:kinsc} reduces to
\begin{equation}
\left\{\begin{array}{l}
M_{\alpha,i}^{n*}=M_{\alpha,i}-\sigma_i\xi\Bigl( M_{\alpha,i+1/2} - M_{\alpha,i-1/2}\Bigr) \\
M_{\alpha,i}^{n+1-} = M_{\alpha,i}^{n*}    +    \Delta t^n ( N_{\alpha+1/2,i}^{n+1-} - N_{\alpha-1/2,i}^{n+1-})
\end{array}\right.
\label{eq:kinsc_st}
\end{equation}
with
\begin{eqnarray}
M_{\alpha,i+1/2} & = &
\1_{\xi>0}M_{\alpha,i}+\1_{\xi<0}M_{\alpha,i+1},\label{eq:Mupwind1}\\
M_{\alpha,i-1/2} & = & \1_{\xi>0}M_{\alpha,i-1}+\1_{\xi<0}M_{\alpha,i}. 
\label{eq:Mupwind2}
\end{eqnarray}
The discrete kinetic equations~\eqref{eq:kinsc_st} allow to precise
the numerical fluxes in~\eqref{eq:upU0} having the form
\begin{equation*}
F_{i+1/2-}= (\sum_{\alpha=1}^N
F_{h_\alpha,i+1/2-},F_{q_1,i+1/2-},\ldots,F_{q_N,i+1/2-})^T,
\label{eq:HRflux_st12}
\end{equation*}
with
\begin{eqnarray*}
	\dsp F_{h_\alpha,i+1/2-} & = & \int_{\R} \xi M_{\alpha,i+1/2} d\xi = \int_{\xi> 0} \xi M_{\alpha,i} d\xi + \int_{\xi< 0} \xi M_{\alpha,i+1} d\xi, \label{eq:HRflux_st1}\\
	\dsp F_{q_\alpha,i+1/2-} & = & \int_{\R} \xi^2 M_{\alpha,i+1/2} d\xi
= \int_{\xi> 0} \xi^2 M_{\alpha,i} d\xi + \int_{\xi< 0} \xi^2
      M_{\alpha,i+1} d\xi.
	\label{eq:HRflux_st2}
\end{eqnarray*}
Note that without topography the flux are conservative since 
$$
F_{i+1/2-}=F_{i+1/2+}
$$
It won't be the case when we will introduce the topography in Section \ref{sec:withtopo}.
The source term $S_i$ in~\eqref{eq:upU0} is defined by
\begin{equation}
S_i = u_{\alpha+1/2,i} G_{\alpha+1/2,i} - u_{\alpha-1/2,i}
G_{\alpha-1/2,i},
\label{eq:Si}
\end{equation}
where $u_{\alpha+1/2,i}$ is defined by~\eqref{eq:upwind_uT} and $G_{\alpha+1/2,i}$
is given by expression~\eqref{eq:Gdis3} that can be rewritten under
the form
\begin{equation*}
\Delta x_i G_{\alpha+1/2,i} = \sum_{j=1}^\alpha \left(
  F_{h_j,i+1/2-} - F_{h_j,i-1/2+} - l_j \sum_{p=1}^N  (F_{h_p,i+1/2-}  - F_{h_p,i-1/2+} )
     \right).
\label{eq:Gdis4}
\end{equation*}


In the following proposition, we prove fundamental stability properties for the
numerical scheme~\eqref{eq:kinsc_st}.

\begin{theorem}
Under the CFL condition
\begin{equation}
\Delta t^n \leq 
\frac{1}{2}
\min_{1 \leq \alpha \leq N}\min_{i \in I}
\frac{\Delta x_i}{|u_{\alpha,i}| + \sqrt{2 g h_i}
 },
\label{eq:CFLfull}
\end{equation}
the scheme \eqref{eq:kinsc_st} satisfies the following properties
\begin{itemize}
\item[\it{(i)}] The kinetic functions remain nonnegative
$M^{n+1-}_{\alpha,i}\geq 0$, $\forall\ \alpha,i,$
\item[\it{(ii)}] One has the kinetic energy equality
\begin{eqnarray}
	H(\overline M_{\alpha,i}^{n+1-},\xi,z_{b,i}) & = & 
        H(\overline{M}_{\alpha,i},\xi,z_{b,i})-\sigma_i\Bigl(\widetilde{H}_{\alpha,i+1/2}
        -  \widetilde{H}_{\alpha,i-1/2}\Bigr) \nonumber\\
& & -\Delta t^n\Bigl(\widehat{H}_{\alpha+1/2,i}^{n+1-}
        -  \widehat{H}_{\alpha-1/2,i}^{n+1-}\Bigr)+ d_{\alpha,i} + e_{\alpha,i},
	\label{eq:entro_st1}
\end{eqnarray}
where $\widetilde{H}_{\alpha,i\pm1/2}$, $\widehat{H}_{\alpha\pm 1/2,i}$ are defined by
\begin{eqnarray*}
\widetilde{H}_{\alpha,i\pm1/2} & = &
\xi H(\overline{M}_{\alpha,i\pm 1/2},\xi,z_{b,i+1/2}) - \xi
H(\overline{M}_{\alpha,i},\xi,z_{b,i}),\\
\widehat{H}_{\alpha\pm 1/2,i}^{n+1-} & = &  \frac{G_{\alpha\pm 1/2,i}}{h_i^{n+1}}H(\overline{M}_{\alpha\pm  1/2,i}^{n+1-},\xi,z_{b,i}),
\end{eqnarray*}
and $d_{\alpha,i}$,~$e_{\alpha,i}$ are given by
\begin{eqnarray}
d_{\alpha,i} & = & \frac{g^2\pi^2}{6}\sigma_i\xi \left( \overline{M}_{\alpha,i+1}
+ 2\overline{M}_{\alpha,i} +\sigma_i\xi(\overline M_{\alpha,i}^{n*} + 2
\overline{M}_{\alpha,i})\right) (\overline{M}_{\alpha,i+1/2} -
\overline{M}_{\alpha,i})^2\\
& & -\frac{g^2\pi^2}{6}\sigma_i\xi \left(\overline{M}_{\alpha,i-1}
+ 2\overline{M}_{\alpha,i} + \sigma_i\xi(\overline M_{\alpha,i}^{n*} + 2
\overline{M}_{\alpha,i})\right) (\overline{M}_{\alpha,i-1/2} -
\overline{M}_{\alpha,i})^2\\
& &  - \frac{g^2\pi^2}{6}\Delta t^n\frac{G_{\alpha+1/2,i}}{h_i ^{n+1}}
(\overline{M}_{\alpha+1/2,i}^{n+1-}+2\overline{M}_{\alpha,i}^{n+1-})(\overline{M}_{\alpha+1/2,i}^{n+1-}-\overline{M}_{\alpha,i}^{n+1-})^2\\
& & +\frac{g^2\pi^2}{6}\Delta t^n\frac{G_{\alpha-1/2,i}}{h_i ^{n+1}}
(\overline{M}_{\alpha-1/2,i}^{n+1-}+2\overline{M}_{\alpha,i}^{n+1-})(\overline{M}_{\alpha-1/2,i}^{n+1-}-\overline{M}_{\alpha,i}^{n+1-})^2\nonumber\\
& & -\frac{g^2\pi^2}{6}(\overline{M}_{\alpha,i}^{n*}
+ 2\overline{M}_{\alpha,i}^{n+1-}) (\overline{M}_{\alpha,i}^{n+1-} -
\overline{M}_{\alpha,i}^{n*})^2,\label{eq:d_i}\\
e_{\alpha,i} & = & \frac{g^2\pi^2}{3} \Delta
t^n\frac{(G_{\alpha+1/2,i}-G_{\alpha-1/2,i})}{h_i^{n+1}}(\overline{M}^{n+1-}_{\alpha,i})^3.
\label{eq:e_i}
\end{eqnarray}
\end{itemize}
\label{thm:entropy_st}
\end{theorem}

\begin{corollary}
Under the CFL condition~\eqref{eq:CFLfull}, one has the macroscopic energy inequality
\begin{equation*}
	\dsp \sum_{\alpha=1}^N \overline E^{n+1}_{\alpha,i} \leq \sum_{\alpha=1}^N \overline E_{\alpha,i} 
	-\sigma_i\Bigl( \sum_{\alpha=1}^N \int_\R \widetilde H_{\alpha,i+1/2}
                  d\xi - \sum_{\alpha=1}^N\int_\R \widetilde H_{\alpha,i-1/2} d\xi\Bigr),
\label{eq:estentrint_st}
\end{equation*}
with following~\eqref{eq:kinmomE}
$$E_{\alpha,i} = l_\alpha \overline E_{\alpha,i} = l_\alpha \int_\R
H(\overline M_{\alpha,i},\xi,z_{b,i}) d\xi.$$
\label{corol:macro_st}
\end{corollary}

\begin{remark}
Even if we consider the system without topography, we keep the
notations $z_{b,j}$, $j=i,i+1/2,i-1/2,\ldots$ so that the obtained
formula can be easily extended to the case of a non flat topography.
\end{remark}

\begin{proof}[Proof of theorem~\ref{thm:entropy_st}]
\noindent {\it (i)} The scheme~\eqref{eq:kin_st1} also writes
\begin{multline}
- \Delta
t^n\frac{|G_{\alpha+1/2,i}|_+}{h_{\alpha+1,i}^{n+1}}
M_{\alpha+1,i}^{n+1-} + \left( 1 - \Delta t^n\frac{|G_{\alpha+1/2,i}|_- -
    |G_{\alpha-1/2,i}|_+}{h_{\alpha,i}^{n+1}}\right)
M_{\alpha,i}^{n+1-}\\
+ \Delta
t^n\frac{|G_{\alpha-1/2,i}|_-}{h_{\alpha-1,i}^{n+1}}
M_{\alpha-1,i}^{n+1-} = M_{\alpha,i}
-\sigma_i\xi\bigl( M_{\alpha,i+1/2} -
M_{\alpha,i-1/2} \bigr).
\label{eq:ineqT}
\end{multline}
Now 
$$
M_{\alpha,i}
-\sigma_i\xi\bigl( M_{\alpha,i+1/2} -
M_{\alpha,i-1/2} \bigr) \geq \bigl(1-\sigma_i |\xi|\bigr)
M_{\alpha,i},
$$
and since $M_{\alpha,i} \geq 0$ the right hand side of \eqref{eq:ineqT} is positive as soon as
$$
\forall \xi \quad 1-\sigma_i |\xi| \geq 0 
$$
that is true under the CFL condition \eqref{eq:CFLfull}. 
Then Property {\it (ii)} of lemma~\ref{lem:inverse} proves {\it (i)}.

\noindent {\it (ii)}  The
proof of Theorem~3.6 in~\cite{JSM_entro} and the proof of the
inequality~\eqref{eq:entro_st1} shares common points, namely the linear dissipation of the
scheme is, in both cases, based on the convexity of the kinetic
entropy~\eqref{eq:kinH} and the form of the
Maxwellian~\eqref{eq:kinmaxw}. But the proof of
inequality~\eqref{eq:entro_st1} is more complex because of the
momentum exchange terms along the vertical axis and their implicit
treatment. Notice that, compared to Theorem~3.6 in~\cite{JSM_entro},
the derivation of the horizontal linear dissipation is obtained in a
different way.

In order to prove~\eqref{eq:entro_st1} we 
will simply multiply the first equation of~\eqref{eq:kinsc_st} by
$\partial_1 H (\overline{M}_{\alpha,i},\xi,z_{b,i})$ and the second equation of~\eqref{eq:kinsc_st} by
$\partial_1 H (\overline{M}_{\alpha,i}^{n+1-},\xi,z_{b,i})$, perform some computations that will take advantage of the kinetic relations and finally add the two relations.

Before to do that, let us first note that, using the identity
\begin{equation}
\overline{M}_{\alpha,i+1}^3 = \overline{M}_{\alpha,i}^3 + 3\overline{M}_{\alpha,i}^2 (\overline{M}_{\alpha,i+1} - \overline{M}_{\alpha,i})
+ (\overline{M}_{\alpha,i+1} + 2\overline{M}_{\alpha,i}) (\overline{M}_{\alpha,i+1} - \overline{M}_{\alpha,i})^2,
\label{eq:cubic}
\end{equation}
we obtain an expression for the linear dissipation associated to the scheme
\begin{equation*}\begin{array}{l}
	\dsp H_\alpha(\overline{M}_{\alpha,i+1},\xi,z_{b,i+1}) - H_\alpha(\overline{M}_{\alpha,i},\xi,z_{b,i})
	-\partial_1 H_\alpha(\overline{M}_{\alpha,i},\xi,z_{b,i})(\overline{M}_{\alpha,i+1}-\overline{M}_{\alpha,i})\\
	\dsp\mkern 100mu
= \frac{g^2\pi^2}{6}(\overline{M}_{\alpha,i+1} + 2\overline{M}_{\alpha,i}) (\overline{M}_{\alpha,i+1} - \overline{M}_{\alpha,i})^2
	\end{array}
\end{equation*}
that can also be written, using definitions~\eqref{eq:Mupwind1}-\eqref{eq:Mupwind2}, under the form
\begin{equation}\begin{array}{l}
	\dsp H_\alpha(\overline{M}_{\alpha,i+1/2},\xi,z_{b,i+1/2}) = H_\alpha(\overline{M}_{\alpha,i},\xi,z_{b,i})
	+\partial_1 H_\alpha(\overline{M}_{\alpha,i},\xi,z_{b,i})(\overline{M}_{\alpha,i+1/2}-\overline{M}_{\alpha,i})\\
	\dsp\mkern 100mu
+ \frac{g^2\pi^2}{6}(\overline{M}_{\alpha,i+1} + 2\overline{M}_{\alpha,i}) (\overline{M}_{\alpha,i+1/2} - \overline{M}_{\alpha,i})^2.
	\label{eq:convleftbis0}
	\end{array}
\end{equation}
A similar expression is obviously available at the interface $i-1/2$
\begin{equation}\begin{array}{l}
	\dsp H_\alpha(\overline{M}_{\alpha,i-1/2},\xi,z_{b,i-1/2}) = H_\alpha(\overline{M}_{\alpha,i},\xi,z_{b,i})
	+\partial_1 H_\alpha(\overline{M}_{\alpha,i},\xi,z_{b,i})(\overline{M}_{\alpha,i-1/2}-\overline{M}_{\alpha,i})\\
	\dsp\mkern 100mu
+ \frac{g^2\pi^2}{6}(\overline{M}_{\alpha,i-1} + 2\overline{M}_{\alpha,i}) (\overline{M}_{\alpha,i-1/2} - \overline{M}_{\alpha,i})^2.
	\label{eq:convleftbis1}
	\end{array}
\end{equation}
Let us now begin by considering the explicit step defined by the first equations of~\eqref{eq:kinsc_st}. 
For $\xi\leq 0$, it writes
\begin{equation}
M_{\alpha,i}^{n*}= M_{\alpha,i} -\sigma_i\xi\bigl(
M_{\alpha,i+1} - M_{\alpha,i}\bigr),
\label{eq:kin_st_up1}
\end{equation}
whereas for $\xi\geq 0$, we have
\begin{equation}
M_{\alpha,i}^{n*}= M_{\alpha,i} -\sigma_i\xi\bigl(
M_{\alpha,i} - M_{\alpha,i-1}\bigr).
\label{eq:kin_st_up2}
\end{equation}
Now let us multiply Eqs.~\eqref{eq:kin_st_up1},\eqref{eq:kin_st_up2} by
$\partial_1 H_\alpha(\overline{M}_{\alpha,i},\xi,z_{b,i})$. Using  
expressions~\eqref{eq:convleftbis0},\eqref{eq:convleftbis1} for the quantities
\begin{align*}
&
  \partial_1 H_\alpha(\overline{M}_{\alpha,i},\xi,z_{b,i})(M_{\alpha,i+1}-M_{\alpha,i})
  = \frac{1}{l_\alpha} \partial_1 H_\alpha(\overline{M}_{\alpha,i},\xi,z_{b,i})(\overline{M}_{\alpha,i+1}-\overline{M}_{\alpha,i}),\\
&
  \partial_1 H_\alpha(\overline{M}_{\alpha,i},\xi,z_{b,i})(M_{\alpha,i}-M_{\alpha,i-1})
  = \frac{1}{l_\alpha} \partial_1 H_\alpha(\overline{M}_{\alpha,i},\xi,z_{b,i})(\overline{M}_{\alpha,i}-\overline{M}_{\alpha,i-1}),
\end{align*}
we obtain the relation
\begin{eqnarray}
l_\alpha \partial_1 H_\alpha(\overline{M}_{\alpha,i},\xi,z_{b,i}) \bigl(
  M_{\alpha,i}^{n*} -  M_{\alpha,i}\bigr) & = &
        -\sigma_i\Bigl(\widetilde{H}_{\alpha,i+1/2}
        - \widetilde{H}_{\alpha,i-1/2}\Bigr)\nonumber\\
& & +\frac{g^2\pi^2}{6}\sigma_i\xi(\overline{M}_{\alpha,i+1}
+ 2\overline{M}_{\alpha,i}) (\overline{M}_{\alpha,i+1/2} -
\overline{M}_{\alpha,i})^2\nonumber\\
& & -\frac{g^2\pi^2}{6}\sigma_i\xi(\overline{M}_{\alpha,i-1}
+ 2\overline{M}_{\alpha,i}) (\overline{M}_{\alpha,i-1/2} -
\overline{M}_{\alpha,i})^2.
\label{eq:entro_nn}
\end{eqnarray}
Then the identities
\begin{align*}
& \partial_1 H (f,\xi,z_b) f = H(f,\xi,z_b) + \frac{\pi^2 g^2}{3}f^3,\\
& \partial_1 H (f,\xi,z_b) \tilde f = H(\tilde f,\xi,z_b) + \frac{\pi^2
  g^2}{2}(f^2 - \tilde f^2) \tilde f + \frac{\pi^2
  g^2}{3} \tilde f^3,
\end{align*}
help us to write, with $\overline M_{\alpha,i}^{n*} =
l_\alpha M_{\alpha,i}^{n*}$,
\begin{eqnarray*}
\partial_1 H_\alpha(\overline{M}_{\alpha,i},\xi,z_{b,i}) \bigl(
  M_{\alpha,i}^{n*} -  M_{\alpha,i}\bigr) & = & \frac{1}{l_\alpha} \partial_1 H_\alpha(\overline{M}_{\alpha,i},\xi,z_{b,i}) \bigl(
  \overline M_{\alpha,i}^{n*} - \overline M_{\alpha,i}\bigr)\\
& = &  \frac{1}{l_\alpha} H_\alpha(\overline M_{\alpha,i}^{n*},\xi,z_{b,i}) -
  \frac{1}{l_\alpha} H_\alpha(\overline{M}_{\alpha,i},\xi,z_{b,i}) - L_{\alpha,i},
\end{eqnarray*}
where 
\begin{equation*}
L_{\alpha,i} = \frac{\pi^2 g^2}{6} (\overline M_{\alpha,i}^{n*} + 2
\overline{M}_{\alpha,i}) (\overline M_{\alpha,i}^{n*} -
\overline{M}_{\alpha,i})^2.
\label{eq:l_alpha}
\end{equation*}
From relations~\eqref{eq:kin_st_up1}-\eqref{eq:kin_st_up2} we can write
$$L_{\alpha,i} = \frac{\pi^2 g^2}{6} (\sigma_i|\xi|)^2(\overline M_{\alpha,i}^{n*} + 2
\overline{M}_{\alpha,i}) \bigl( (\overline{M}_{\alpha,i+1} -
\overline{M}_{\alpha,i})^2\1_{\xi\leq 0} + (\overline{M}_{\alpha,i} -
\overline{M}_{\alpha,i-1})^2\1_{\xi\geq 0}\bigr).$$
Therefore we are able to write~\eqref{eq:entro_nn} under the form
\begin{eqnarray}
H_\alpha(\overline M_{\alpha,i}^{n*},\xi,z_{b,i}) & = & H_\alpha(\overline{M}_{\alpha,i},\xi,z_{b,i}) 
        -\sigma_i\Bigl(\widetilde{H}_{\alpha,i+1/2}
        - \widetilde{H}_{\alpha,i-1/2}\Bigr)\nonumber\\
& & +l_\alpha\frac{g^2\pi^2}{6}\sigma_i\xi \bigl(\overline{M}_{\alpha,i+1}
+ 2\overline{M}_{\alpha,i} + \sigma_i\xi(\overline M_{\alpha,i}^{n*} + 2
\overline{M}_{\alpha,i})\bigr) (\overline{M}_{\alpha,i+1/2} -
\overline{M}_{\alpha,i})^2\nonumber\\
& & -l_\alpha\frac{g^2\pi^2}{6}\sigma_i\xi \bigl(\overline{M}_{\alpha,i-1}
+ 2\overline{M}_{\alpha,i} + \sigma_i\xi(\overline M_{\alpha,i}^{n*} + 2
\overline{M}_{\alpha,i})\bigr) (\overline{M}_{\alpha,i-1/2} -
\overline{M}_{\alpha,i})^2.
\label{eq:entrfullystat1_JSM}
\end{eqnarray}
We now consider the implicit part of the scheme. We then multiply the second equation of~\eqref{eq:kinsc_st} by $\partial_1 H_\alpha(\overline{M}_{\alpha,i}^{n+1-},\xi,z_{b,i})$. 
Proof of Proposition~\ref{prop:entropy_kin} allows us to write 
\begin{equation*}\begin{array}{l}
\partial_1 H_\alpha(\overline{M}_{\alpha,i}^{n+1-},\xi,z_{b,i})
(N_{\alpha+1/2,i}^{n+1-} - N_{\alpha-1/2,i}^{n+1-}) =\\
\dsp \mkern 100mu                   \frac{G_{\alpha+
                   1/2,i}}{h^{n+1}}H_\alpha(\overline{M}_{\alpha+
                   1/2,i}^{n+1-},\xi,z_{b,i}) - \frac{G_{\alpha- 1/2,i}}{h^{n+1}}H_\alpha(\overline{M}_{\alpha-  1/2,i}^{n+1-},\xi,z_{b,i})  \\
\dsp \mkern 100mu -l_\alpha\frac{g^2\pi^2}{6}\frac{G_{\alpha+1/2,i}}{h_i ^{n+1}}
(\overline{M}_{\alpha+1/2,i}^{n+1-}+2\overline{M}_{\alpha,i}^{n+1-})(\overline{M}_{\alpha+1/2,i}^{n+1-}-\overline{M}_{\alpha,i}^{n+1-})^2\\
\dsp \mkern 100mu +l_\alpha\frac{g^2\pi^2}{6}\frac{G_{\alpha-1/2,i}}{h_i ^{n+1}}
(\overline{M}_{\alpha-1/2,i}^{n+1-}+2\overline{M}_{\alpha,i}^{n+1-})(\overline{M}_{\alpha-1/2,i}^{n+1-}-\overline{M}_{\alpha,i}^{n+1-})^2\\
\dsp \mkern 100mu + l_\alpha\frac{g^2\pi^2}{3} \frac{(G_{\alpha+1/2,i}-G_{\alpha-1/2,i})}{h_i ^{n+1}}(\overline{M}^{n+1-}_{\alpha,i})^3.
\end{array}\end{equation*}
But we also have
\begin{multline*}
\partial_1 H_\alpha(\overline{M}_{\alpha,i}^{n+1-},\xi,z_{b,i}) \bigl(\overline{M}_{\alpha,i}^{n+1-}
- \overline{M}_{\alpha,i}^{n*} \bigr) =
H_\alpha(\overline{M}_{\alpha,i}^{n+1-},\xi,z_{b,i}) -
H_\alpha(\overline{M}_{\alpha,i}^{n*},\xi,z_{b,i})\\
-l_\alpha\frac{g^2\pi^2}{6}(\overline{M}_{\alpha,i}^{n*}
+ 2\overline{M}_{\alpha,i}^{n+1-}) (\overline{M}_{\alpha,i}^{n+1-} -
\overline{M}_{\alpha,i}^{n*})^2.
\end{multline*}
Using the two previous expressions, we are able to rewrite the second equation of~\eqref{eq:kinsc_st} multiplied by $\partial_1 H_\alpha(\overline{M}_{\alpha,i}^{n+1-},\xi,z_{b,i})$ under the form
\begin{eqnarray}
	H_\alpha(\overline M_{\alpha,i}^{n+1-},\xi,z_{b,i}) & = &
        H_\alpha(\overline{M}_{\alpha,i}^{n*},\xi,z_{b,i})-\Delta t\Bigl(\widehat{H}_{\alpha+1/2,i}^{n+1-}
        - \widehat{H}_{\alpha-1/2,i}^{n+1-}\Bigr)\nonumber\\
& &  - l_\alpha\frac{g^2\pi^2}{6}\Delta t^n\frac{G_{\alpha+1/2,i}}{h_i ^{n+1}}
(\overline{M}_{\alpha+1/2,i}^{n+1-}+2\overline{M}_{\alpha,i}^{n+1-})(\overline{M}_{\alpha+1/2,i}^{n+1-}-\overline{M}_{\alpha,i}^{n+1-})^2\nonumber\\
& & +l_\alpha\frac{g^2\pi^2}{6}\Delta t^n\frac{G_{\alpha-1/2,i}}{h_i ^{n+1}}
(\overline{M}_{\alpha-1/2,i}^{n+1-}+2\overline{M}_{\alpha,i}^{n+1-})(\overline{M}_{\alpha-1/2,i}^{n+1-}-\overline{M}_{\alpha,i}^{n+1-})^2\nonumber\\
& & + l_\alpha\frac{g^2\pi^2}{3} \Delta
t^n\frac{(G_{\alpha+1/2,i}-G_{\alpha-1/2,i})}{h_i
    ^{n+1}}(\overline{M}_{\alpha,i}^{n+1-})^3\nonumber\\
& & -l_\alpha\frac{g^2\pi^2}{6}(\overline{M}_{\alpha,i}^{n*}
+ 2\overline{M}_{\alpha,i}^{n+1-}) (\overline{M}_{\alpha,i}^{n+1-} -
\overline{M}_{\alpha,i}^{n*})^2.
	\label{eq:entrfullystat2_JSM}
\end{eqnarray}
The sum of relations~\eqref{eq:entrfullystat1_JSM},
\eqref{eq:entrfullystat2_JSM} divided by $l_\alpha$
gives the result.
\end{proof}


\begin{proof}[Proof of corollary~\ref{corol:macro_st}]
An integration in $\xi$ of
relation~\eqref{eq:entro_st1} and a sum of the obtained relation for
$\alpha=1,\ldots,N$ gives
\begin{multline*}
	\dsp \sum_{\alpha=1}^N \overline E^{n+1-}_{\alpha,i} = \sum_{\alpha=1}^N \overline E_{\alpha,i} 
	-\sigma_i\Bigl( \sum_{\alpha=1}^N \int_\R \widetilde H_{\alpha,i+1/2}
                  d\xi - \sum_{\alpha=1}^N\int_\R \widetilde H_{\alpha,i-1/2} d\xi\Bigr)\\
+ \sum_{\alpha=1}^N \int_\R  d_{\alpha,i} d\xi + \sum_{\alpha=1}^N \int_\R e_{\alpha,i}\ d\xi,
\end{multline*}
where $d_{\alpha,i}$ and $e_{\alpha,i}$ are defined
by~\eqref{eq:d_i},~\eqref{eq:e_i}.

Using~\eqref{eq:kin_st_up1}, the first line of $d_{\alpha,i}$ writes
$$\frac{g^2\pi^2}{6}\sigma_i\xi \left( (1-(\sigma_i\xi)^2)\overline{M}_{\alpha,i+1}
+ \bigl(2 + 3 \sigma_i\xi - (\sigma_i\xi)^2\big) \overline{M}_{\alpha,i} \right) (\overline{M}_{\alpha,i+1/2} -
\overline{M}_{\alpha,i})^2,$$
that is non positive under the CFL condition~\eqref{eq:CFLfull}. Likewise, we obtain
that the second line of $d_{\alpha,i}$ is non positive. Moreover, it is
obvious from the definition of $G_{\alpha+1/2,i}$ given
by~\eqref{eq:upwind_uT} that the other lines of $d_{\alpha,i}$ are also non positive.

It remains to study the quantity
\begin{equation}
\sum_{\alpha=1}^N \int_\R e_{\alpha,i}\ d\xi = \sum_{\alpha=1}^N \int_\R \frac{g^2\pi^2}{3} \Delta
t^n\frac{(G_{\alpha+1/2,i}-G_{\alpha-1/2,i})}{h_i
    ^{n+1}}(\overline{M}_{\alpha,i}^{n+1-})^3 d\xi.
\label{eq:sum_e}
\end{equation}
Since $\overline{M}_{\alpha,i}^{n+1-}$ is not a Maxwellian, it is not
possible to conclude, as in the proof of Proposition~\ref{prop:entropy_kin}, that
$$\int_\R (\overline{M}_{\alpha,i}^{n+1-})^3 d\xi = \frac{g}{2}h_{\alpha,i}^2,$$
and that the sum~\eqref{eq:sum_e} is zero. So we proceed as follows.

Let us rewrite Eq.~\eqref{eq:kinsc_st} under the
equivalent form
\begin{equation*}
\left\{\begin{array}{l}
M_{\alpha,i}^{n\$}=M_{\alpha,i}-\sigma_i\Bigl( \xi M_{\alpha,i+1/2}
         - \xi M_{\alpha,i-1/2} 
	\Bigr) + \Delta t^n Q_{\alpha,i}\\
M_{\alpha,i}^{n+1} = M_{\alpha,i}^{n\$}    +    \Delta t^n (
         N_{\alpha+1/2,i}^{n+1} - N_{\alpha-1/2,i}^{n+1}) + \Delta t^n Q_{\alpha,i}^{n\$}
\end{array}\right.
\end{equation*}
where $Q_{\alpha,i}$, $Q_{\alpha,i}^{n\$}$ are two collision terms
satisfying the integral relations \eqref{eq:collisionbis}.
Arguments of the proof of 
Theorem~\ref{thm:entropy_st}-{\it (ii)} remain unchanged 
where the superscript $^{n*}$
(resp. $^{n+1-}$) becomes $^{n\$}$ (resp. $^{n+1}$) and the obvious equalities
$$\int_\R \partial_1 H_\alpha(\overline{M}_{\alpha,i},\xi,z_{b,i})
Q_{\alpha,i} d\xi =\int_\R \partial_1 H_\alpha(\overline{M}_{\alpha,i}^{n+1},\xi,z_{b,i})
Q_{\alpha,i}^{n\$} d\xi = 0.$$
Since the quantity $\overline{M}_{\alpha,i}^{n+1}$ is now a Maxwellian, we can conclude that
\begin{multline*}
\sum_{\alpha=1}^N \int_\R e_{\alpha,i}\ d\xi = \sum_{\alpha=1}^N \int_\R \frac{g^2\pi^2}{3} \Delta
t^n\frac{(G_{\alpha+1/2,i}-G_{\alpha-1/2,i})}{h_i
    ^{n+1}}(\overline{M}_{\alpha,i}^{n+1-})^3 d\xi \\
= \sum_{\alpha=1}^N \int_\R \frac{g^2\pi^2}{3} \Delta
t^n\frac{(G_{\alpha+1/2,i}-G_{\alpha-1/2,i})}{h_i
    ^{n+1}}(\overline{M}_{\alpha,i}^{n+1})^3 d\xi = 0.
\end{multline*}
Let us note that, since $\overline{M}_{\alpha,i}^{n+1}$ is by definition the Maxwellian with the same moments that the density function $\overline{M}_{\alpha,i}^{n+1-}$, Lemma \ref{lemma:energy} gives us
$$
\dsp \sum_{\alpha=1}^N \overline E^{n+1}_{\alpha,i} \leq \dsp \sum_{\alpha=1}^N \overline E^{n+1-}_{\alpha,i}
$$
that concludes the proof.
\end{proof}

\subsection{With topography}
\label{sec:withtopo}
In this paragraph we examine the properties of the discrete
scheme~\eqref{eq:kinsc} when the topography source term is no more neglected.

The hydrostatic reconstruction scheme (HR scheme for short) is a general method giving, with any solver, a
robust and efficient discretization of the source terms in
conservation laws. It has been initially proposed for the Saint-Venant
system leading to a consistent,
well-balanced, positive scheme satisfying a semi-discrete entropy
inequality~\cite{bristeau1}. Here we use the HR technique to discretize
the topography source term appearing in~\eqref{eq:eq44} and we prove
the kinetic scheme~\eqref{eq:kinsc} coupled with the HR technique leads to a consistent,
well-balanced, positive scheme satisfying a fully discrete entropy
inequality with a controlled error term. It generalizes to the layerwise framework the result obtained in \cite{JSM_entro} for the classical shallow water model that was used in~\cite{lhebrard} to demonstrate the convergence of the scheme.

With first briefly recall the main features of the HR technique. The HR scheme uses reconstructed states
\begin{equation}
	U_{\alpha,i+1/2-}=(h_{\alpha,i+1/2-},h_{\alpha,i+1/2-}u_i),\qquad
	U_{\alpha,i+1/2+}=(h_{\alpha,i+1/2+},h_{\alpha,i+1/2+}u_{i+1}),
	\label{eq:lrstates}
\end{equation}
defined by
\begin{equation}
\begin{array}{l}
	h_{i+1/2-}=(h_i+z_i-z_{b,i+1/2})_+,\\
h_{i+1/2+}=(h_{i+1}+z_{b,i+1}-z_{b,i+1/2})_+,\\
        h_{\alpha,i+1/2\pm} = l_\alpha h_{i+1/2\pm},\\
        M_{\alpha,i+1/2\pm}=M(U_{\alpha,i+1/2\pm},\xi)
\end{array}
	\label{eq:hlr}
\end{equation}
and
\begin{equation}
	z_{b,i+1/2}=\max(z_{b,i},z_{b,i+1}).
	\label{eq:zstar}
\end{equation}
We note that the definitions of $h_{i+1/2\pm}$ in \eqref{eq:hlr}-\eqref{eq:zstar} ensure
that $h_{i+1/2-}\leq h_i$, and $h_{i+1/2+}\leq h_{i+1}$. Now we can transfer these results to the kinetic level. First, 
because of \eqref{eq:kinmaxw}, one has
\begin{equation}
	0\leq M_{\alpha,i+1/2-}\leq M_{\alpha,i},\quad 0\leq M_{\alpha,i+1/2+}\leq M_{\alpha,i+1},
	\label{eq:domM_i+1/2}
\end{equation}
and thus
\begin{equation*}
	M(U_{\alpha,i},\xi)=0\ \Rightarrow M(U_{\alpha,i+1/2-},\xi)=0\mbox{ and }M(U_{\alpha,i-1/2+},\xi)=0.
	\label{eq:supporti+1/2}
\end{equation*}
Let us now consider the kinetic source terms $\delta M_{\alpha,i+1/2\pm}$. They depend on $\xi$, $U_{\alpha,i}$, $U_{\alpha,i+1}$, $\Delta z_{i+1/2}=z_{i+1}-z_i$,
and satisfy the moment relations
\begin{equation}
	\int_\R\delta M_{\alpha,i+1/2-}\,d\xi=0,\quad
	\int_\R\xi\,\delta M_{\alpha,i+1/2-}\,d\xi=g\frac{h_{\alpha,i}^2}{2}-g\frac{h_{\alpha,i+1/2-}^2}{2},
	\label{eq:intdeltaM-}
\end{equation}
\begin{equation}
	\int_\R\delta M_{\alpha,i-1/2+}\,d\xi=0,\quad
	\int_\R\xi\,\delta M_{\alpha,i-1/2+}\,d\xi=g\frac{h_{\alpha,i}^2}{2}-g\frac{h_{\alpha,i-1/2+}^2}{2}.
	\label{eq:intdeltaM+}
\end{equation}
We also assume that,
\begin{equation}
	M(U_{\alpha,i},\xi)=0\ \Rightarrow \delta M_{\alpha,i+1/2-}(\xi)=0\mbox{ and }\delta M_{\alpha,i-1/2+}(\xi)=0.
	\label{eq:support_deltai+1/2}
\end{equation}
For reasons that will appear later during the derivation of the
entropy inequality, we make the choice
\begin{equation}\begin{array}{l}
	\dsp \delta M_{\alpha,i+1/2-}=(\xi-u_{\alpha,i})(M_{\alpha,i}-M_{\alpha,i+1/2-}),\\
	\dsp \delta M_{\alpha,i-1/2+}=(\xi-u_{\alpha,i})(M_{\alpha,i}-M_{\alpha,i-1/2+}),
	\label{eq:defdeltaM}
	\end{array}
\end{equation}
that satisfies the assumptions \eqref{eq:intdeltaM-}, \eqref{eq:intdeltaM+}
and \eqref{eq:support_deltai+1/2}. This allows to precise
the numerical fluxes in~\eqref{eq:upU0} having the form
\begin{equation}
F_{i+1/2-}= (\sum_{\alpha=1}^N
F_{h_\alpha,i+1/2-},F_{q_1,i+1/2-},\ldots,F_{q_N,i+1/2-})^T,
\label{eq:HRflux12}
\end{equation}
with
\begin{eqnarray}
	\dsp F_{h_\alpha,i+1/2-} & = &\int_{\R} \xi M_{\alpha,i+1/2}
                                       d\xi = \int_{\xi> 0} \xi
                                       M_{\alpha,i+1/2-} d\xi +
                                       \int_{\xi< 0} \xi
                                       M_{\alpha,i+1/2+} d\xi,
                                       \nonumber\\ 
	\dsp F_{q_\alpha,i+1/2-} & = & \int_{\R} (\xi^2 M_{\alpha,i+1/2} +
  \xi \delta M_{\alpha,i+1/2-}) d\xi + 
  \nonumber\\
& = & \int_{\xi> 0} \xi^2
  M_{\alpha,i+1/2-} d\xi + \int_{\xi< 0} \xi^2 M_{\alpha,i+1/2+} d\xi
+ \frac{gh_{\alpha,i}^2}{2}-\frac{gh_{\alpha,i+1/2-}^2}{2}.
	\label{eq:HRflux2}
\end{eqnarray}
The source term $S_i$ remains unchanged, see \eqref{eq:Si}, since the topography source term is taken into account in the (now non conservative) fluxes \eqref{eq:HRflux2}.

Now we prove some properties of the scheme~\eqref{eq:kinsc}  with the choice~\eqref{eq:defdeltaM}.
Notice that only the explicit part of the scheme~\eqref{eq:kinsc}, i.e relation \eqref{eq:kinsc1}, has been affected by the topography. The implicit part is unchanged and still requires to
invert the matrix ${\bf I}_N + \Delta t {\bf G}_{N,i}$ whose
properties have already been studied in lemma~\ref{lem:inverse}. In particular the result of Remark~\ref{rem:G_bound_dis} concerning the boundedness of the
quantities ${G_{\alpha\pm 1/2,i}}/{h_i^{n+1}}$
remains valid.



\begin{theorem}
Under the CFL condition
\begin{equation}
\Delta t^n < \min_{1 \leq \alpha \leq N}\min_{i \in I}
\frac{\Delta x_i}{|u_{\alpha,i}| + 2\sqrt{2 g h_i}
 }, 
	\label{eq:CFLfull_topo}
\end{equation}
the scheme \eqref{eq:kinsc} with the choice~\eqref{eq:defdeltaM} verifies the following properties.

\noindent (i) The kinetic functions remain nonnegative $M^{n+1-}_{\alpha,i}\geq 0$.

\noindent (ii) The scheme \eqref{eq:kinsc} is kinetic well-balanced.

\noindent (iii) One has the kinetic relation
\begin{multline}
	H(\overline M_{\alpha,i}^{n+1-},\xi,z_{b,i}) = H(\overline M_{\alpha,i},\xi,z_{b,i})-\sigma_i\Bigl(\widetilde{H}_{\alpha,i+1/2-} -
        \widetilde{H}_{\alpha,i-1/2+}\Bigr) \\
-\Delta t^n\Bigl(\widehat{H}_{\alpha+1/2,i}^{n+1-} -
        \widehat{H}_{\alpha-1/2,i}^{n+1-}\Bigr)+ d_{\alpha,i} + e_{\alpha,i},
	\label{eq:entrfullystat3_JSM}
\end{multline}
where
\begin{equation}\begin{array}{l}
	\dsp\widetilde H_{i+1/2-}=\xi\1_{\xi<0}H(M_{i+1/2+},z_{i+1/2})+\xi\1_{\xi>0}H(M_{i+1/2-},z_{i+1/2})\\
	\dsp\hphantom{\widetilde H_{i+1/2-}=}+\xi H(M_i,z_i)-\xi H(M_{i+1/2-},z_{i+1/2})\vphantom{\Bigl|}\\
	\dsp\hphantom{\widetilde H_{i+1/2-}=}+\Bigl(\eta'(U_i)\kxi+gz_i\Bigr)
	\bigl(\xi M_{i+1/2-}-\xi M_i+\delta M_{i+1/2-}\bigr),
	\end{array}
	\label{eq:tildH-}
\end{equation}
\begin{equation}\begin{array}{l}
	\dsp\widetilde H_{i-1/2+}=\xi\1_{\xi<0}H(M_{i-1/2+},z_{i-1/2})+\xi\1_{\xi>0}H(M_{i-1/2-},z_{i-1/2})\\
	\dsp\hphantom{\widetilde H_{i+1/2-}=}+\xi H(M_i,z_i)-\xi H(M_{i-1/2+},z_{i-1/2})\vphantom{\Bigl|}\\
	\dsp\hphantom{\widetilde H_{i+1/2-}=}+\Bigl(\eta'(U_i)\kxi+gz_i\Bigr)
	\bigl(\xi M_{i-1/2+}-\xi M_i+\delta M_{i-1/2+}\bigr).
	\end{array}
	\label{eq:tildH+}
\end{equation}
and $\widehat{H}_{\alpha+1/2,i}^{n+1-}$,
$\widehat{H}_{\alpha-1/2,i}^{n+1-}$ are defined in
Theorem~\ref{thm:entropy_st}. The terms $d_{\alpha,i}$,~$e_{\alpha,i}$ satisfy the estimates
\begin{eqnarray*}
d_{\alpha,i} & \leq &
\sigma_i\xi\frac{g^2\pi^2}{6}\left(\overline{M}_{\alpha,i+1/2+} +
  2\overline{M}_{\alpha,i+1/2-}+\sigma_i\xi (2\overline{M}_{\alpha,i}+\overline M^{n*}_{\alpha,i}) \right)
(\overline{M}_{\alpha,i+1/2} - \overline{M}_{\alpha,i+1/2-})^2\nonumber\\
& & 
-\sigma_i\xi\frac{g^2\pi^2}{6}\left(\overline{M}_{\alpha,i-1/2-} +
  2\overline{M}_{\alpha,i-1/2+}-\sigma_i\xi (2\overline{M}_{\alpha,i}+\overline M^{n*}_{\alpha,i})\right) (\overline{M}_{\alpha,i-1/2} -
\overline{M}_{\alpha,i-1/2+})^2\nonumber\\
& & - \Delta t^n \frac{g^2\pi^2}{6} \frac{G_{\alpha+1/2,i}}{h_i^{n+1}}
(\overline{M}_{\alpha+1/2,i}^{n+1-}+2\overline{M}_{\alpha,i}^{n+1-})(\overline{M}_{\alpha+1/2,i}^{n+1-}-\overline{M}_{\alpha,i}^{n+1-})^2\nonumber\\
& & + \Delta t^n \frac{g^2\pi^2}{6} \frac{G_{\alpha-1/2,i}}{h_i^{n+1}} (\overline{M}_{\alpha-1/2,i}^{n+1-}+2\overline{M}_{\alpha,i}^{n+1})(\overline{M}_{\alpha-1/2,i}^{n+1-}-\overline{M}_{\alpha,i}^{n+1-})^2,\\
e_{\alpha,i} & \leq & \sigma_i^2 \frac{g^2\pi^2}{3} u_{\alpha,i}^2
(2\overline{M}_{\alpha,i}+\overline M^{n*}_{\alpha,i})(\overline{M}_{\alpha,i-1/2+}^{n+1-}
- \overline{M}_{\alpha,i+1/2-}^{n+1-})^2\nonumber\\
& & + \Delta t^n
\frac{g^2\pi^2}{3}\frac{(G_{\alpha+1/2,i}-G_{\alpha-1/2,i})}{h_i ^{n+1}}(\overline{M}_{\alpha,i}^{n+1-})^3.
\end{eqnarray*}
\label{thm:entropy_discrete}
\end{theorem}

\begin{remark}
Notice that the integral with respect to $\xi$ of the last two lines of \eqref{eq:tildH-} (respectively 
of \eqref{eq:tildH+}) vanishes and this will be used in the
Corollary~\ref{corol:macro} to establish the macroscopic energy inequality~\eqref{eq:estentrint}.
\end{remark}

\begin{remark}
The CFL condition \eqref{eq:CFLfull_topo} is a bit less restrictive than the CFL condition \eqref{eq:CFLfull}. It is because here, we do not need to prove the nonpositivity of terms $d_{\alpha,i}$ in relation \eqref{eq:entrfullystat3_JSM} but the nonpositivity of terms $d^1_{\alpha,i}$ in relation \eqref{eq:entrfullystat4_JSM}. Indeed we will prove in Corollaries \ref{corol:entropy_discrete_estim} and \ref{corol:macro} a slightly different entropy inequality 	\eqref{eq:estentrint} that now contains an error term that is proved to be controled. Note that the CFL condition \eqref{eq:CFLfull_topo} can also be written as
\begin{equation}
\sigma_i v_m \leq \beta, \qquad v_m = \max_{1 \leq \alpha \leq N}\max_{i \in I} (|u_{\alpha,i}| + 2\sqrt{2 g h_i}), \qquad \beta < 1.
\label{CFLv2}
\end{equation}
\end{remark}


Similar estimates have been obtained
in~\cite{JSM_entro} in the context of the classical Saint-Venant
system and using the same arguments as
in~\cite[Theorem~3.6]{JSM_entro}, the following corollaries hold.

\begin{corollary}
Under the CFL condition~\eqref{eq:CFLfull_topo}, the scheme
\eqref{eq:kinsc} with the choice~\eqref{eq:defdeltaM} leads to the kinetic
entropy inequality
\begin{multline}
	H(\overline M_{\alpha,i}^{n+1-},\xi,z_{b,i}) = H(\overline M_{\alpha,i},\xi,z_{b,i})-\sigma_i\Bigl(\widetilde{H}_{\alpha,i+1/2-} -
        \widetilde{H}_{\alpha,i-1/2+}\Bigr) \\
-\Delta t^n\Bigl(\widehat{H}_{\alpha+1/2,i}^{n+1-} -
        \widehat{H}_{\alpha-1/2,i}^{n+1-}\Bigr)+ d^1_{\alpha,i} + e^1_{\alpha,i},
	\label{eq:entrfullystat4_JSM}
\end{multline}
where $\widetilde{H}_{\alpha,i+1/2-}$ ,~$\widetilde{H}_{\alpha,i-1/2+}$ are defined in Theorem~\ref{thm:entropy_discrete}
and $\widehat{H}_{\alpha+1/2,i}^{n+1-}$, $\widehat{H}_{\alpha-1/2,i}^{n+1-}$ are defined
in Theorem~\ref{thm:entropy_st}.
The quantities $d^1_{\alpha,i}$, $e^1_{\alpha,i}$ satisfy
\begin{eqnarray*}
d^1_{\alpha,i} & \leq &
\nu_\beta\sigma_i\xi\frac{g^2\pi^2}{6}\left(\overline{M}_{\alpha,i+1/2+} +
  \overline{M}_{\alpha,i+1/2-} \right)
(\overline{M}_{\alpha,i+1/2} - \overline{M}_{\alpha,i+1/2-})^2\nonumber\\
& & 
-\nu_\beta\sigma_i\xi\frac{g^2\pi^2}{6}\left(\overline{M}_{\alpha,i-1/2-} +
  \overline{M}_{\alpha,i-1/2+})\right) (\overline{M}_{\alpha,i-1/2} -
\overline{M}_{\alpha,i-1/2+})^2\nonumber\\
& & - \Delta t^n \frac{g^2\pi^2}{6} \frac{G_{\alpha+1/2,i}}{h_i^{n+1}}
(\overline{M}_{\alpha+1/2,i}^{n+1-}+2\overline{M}_{\alpha,i}^{n+1-})(\overline{M}_{\alpha+1/2,i}^{n+1-}-\overline{M}_{\alpha,i}^{n+1-})^2\nonumber\\
& & + \Delta t^n \frac{g^2\pi^2}{6} \frac{G_{\alpha-1/2,i}}{h_i^{n+1}} (\overline{M}_{\alpha-1/2,i}^{n+1-}+2\overline{M}_{\alpha,i}^{n+1})(\overline{M}_{\alpha-1/2,i}^{n+1-}-\overline{M}_{\alpha,i}^{n+1-})^2,
\label{eq:d_i_topo}\\
e^1_{\alpha,i} & \leq & C_\beta(\sigma_iv_m)^2\frac{g^2 \pi^2}{6}
	\overline M_i\Bigl((\overline M_{\alpha,i}-\overline
                      M_{\alpha,i+1/2-})^2+(\overline M_{\alpha,i} -\overline M_{\alpha,i-1/2+})^2\Bigr)\nonumber\\
& & + \Delta t^n
\frac{g^2\pi^2}{3}\frac{(G_{\alpha+1/2,i}-G_{\alpha-1/2,i})}{h_i ^{n+1}}(\overline{M}_{\alpha,i}^{n+1-})^3.
\end{eqnarray*}
where $\nu_\beta>0$ is a dissipation constant
depending only on $\beta$, see relation \eqref{CFLv2}, and $C_\beta\geq 0$ is a constant depending only on $\beta$.
The term proportional to $C_\beta$ is an error term, while the term
proportional to $\nu_\beta$ is a dissipation term that reinforces the
inequality.
\label{corol:entropy_discrete_estim}
\end{corollary}

\begin{corollary}
Under the CFL condition~\eqref{eq:CFLfull_topo}, integrating the
relation~\eqref{eq:entrfullystat4_JSM} with respect to $\xi$ and
summing for $\alpha=1,\ldots,N$,
yields that
\begin{eqnarray}
	\dsp \sum_{\alpha=1}^N \overline E^{n+1}_{\alpha,i} & \leq & \sum_{\alpha=1}^N \overline E_{\alpha,i} 
	-\sigma_i\Bigl( \sum_{\alpha=1}^N \int_\R \widetilde H_{\alpha,i+1/2}
                  d\xi - \int_\R \widetilde H_{\alpha,i-1/2} d\xi\Bigr)\nonumber\\
& &  +C_\beta (\sigma_iv_m)^2\biggl(g(h_i-h_{i+1/2-})^2+g(h_i-h_{i-1/2+})^2 \biggr).
	\label{eq:estentrint}
\end{eqnarray}
As in~\cite[Corollary~3.7]{JSM_entro}, we conclude that relation~\eqref{eq:estentrint} is the discrete entropy inequality associated to the HR scheme~\eqref{eq:upU0},\eqref{eq:hlr},\eqref{eq:zstar}
with kinetic numerical flux~\eqref{eq:HRflux12}-\eqref{eq:HRflux2}.
With \eqref{eq:lrstates}-\eqref{eq:zstar} one has
\begin{equation*}
	0\leq h_i-h_{i+1/2-}\leq|z_{b,i+1}-z_{b,i}|,\quad 0\leq h_i-h_{i-1/2+}\leq |z_{b,i}-z_{b,i-1}|.
	\label{eq:errestexpl}
\end{equation*}
We conclude that the quadratic error terms proportional to $C_\beta$ in the right-hand side of \eqref{eq:estentrint}
(divide \eqref{eq:estentrint} by $\Delta t^n$ to be consistent with \eqref{eq:energy_euler_eq})
has the following key properties: it vanishes identically when $z_b=cst$ (no topography)
or when $\sigma_i\rightarrow 0$ (semi-discrete limit), and as soon as the topography
is Lipschitz continuous, it tends to zero strongly when the grid size tends to $0$
(consistency with the continuous entropy inequality \eqref{eq:energy_euler_eq}),
even if the solution contains shocks.
\label{corol:macro}
\end{corollary}

\begin{proof}[Proof of theorem~\ref{thm:entropy_discrete}]
\noindent (i) 
The proof is very similar to the one of Theorem \ref{thm:entropy_st}, Item (i) but
the right hand side in \eqref{eq:ineqT} is now 
$$
M_{\alpha,i}
-\sigma_i\bigl( \xi M_{\alpha,i+1/2} + \delta M_{\alpha,i+1/2} -
M_{\alpha,i-1/2} - \delta M_{\alpha,i-1/2}\bigr) \geq 
 \bigl(1-\sigma_i (|\xi| + |u_{\alpha,i}|)\bigr) M_{\alpha,i}
 $$
that can be proved to be positive under the CFL condition ~\eqref{eq:CFLfull_topo}.

\noindent (ii) When $u_{\alpha,i} =0$, $h_i + z_{b,i}=cst$ for any $\alpha,i$ then
for any $\xi$ we have $M_{\alpha,i+1/2+}=M_{\alpha,i+1/2-}=M_{\alpha,i-1/2+}=M_{\alpha,i-1/2-}$, $G_{\alpha+1/2,i}=G_{\alpha-1/2,i}=0$  and therefore
$M^{n+1-}_{\alpha,i} = M_{\alpha,i}$ proving (ii).

\noindent (iii) In order to prove (iii) we proceed as in the proof of
Theorem~\ref{thm:entropy_st}, item (ii) but the computations are more complex because of the
topography source terms. The implicit part has not been modified and then, by multiplying~\eqref{eq:kinsc2}, we still get the relation \eqref{eq:entrfullystat2_JSM}.
The complexity lies in the explicit part. Let us multiply~\eqref{eq:kinsc1}, with topography terms defined by~\eqref{eq:defdeltaM}, by $\partial_1 H (\overline{M}_{\alpha,i},\xi,z_{b,i})$. After computations that are similar to what we did to prove Theorem~\ref{thm:entropy_st}, we get
\begin{multline}
H(\overline M^{n*}_{\alpha,i},\xi,z_{b,i})  = 
H(\overline{M}_{\alpha,i},\xi,z_{b,i}) \\
-\sigma_i\xi\bigl(
H(\overline{M}_{\alpha,i+1/2},\xi,z_{b,i+1/2}) -
    H(\overline{M}_{\alpha,i-1/2},\xi,z_{b,i-1/2})\bigr) +
    R^x_{\alpha,i} + R^t_{\alpha,i},
\label{eq:cindis01_d0}
\end{multline}
where $R^x_{\alpha,i}$ (resp. $R^t_{\alpha,i}$) is an error term
coming from the space (resp. time)
discretization 
\begin{eqnarray*}
R^x_{\alpha,i} & = & \sigma_i\xi \left( H(\overline{M}_{\alpha,i+1/2},\xi,z_{b,i+1/2}) -
  \partial_1 H (\overline{M}_{\alpha,i},\xi,z_{b,i})\overline{M}_{\alpha,i+1/2}\right)\nonumber\\
& & -\sigma_i\xi \left( H(\overline{M}_{\alpha,i-1/2},\xi,z_{b,i-1/2}) -
  \partial_1 H (\overline{M}_{\alpha,i},\xi,z_{b,i})\overline{M}_{\alpha,i-1/2}\right)\nonumber\\
& & + \sigma_i (\xi - u_{\alpha,i})
\partial_1 H (\overline{M}_{\alpha,i},\xi,z_{b,i})(\overline{M}_{\alpha,i+1/2-}
- \overline{M}_{\alpha,i-1/2+}),\\
R^t_{\alpha,i} & = & \frac{\pi^2 g^2}{6} (\overline M_{\alpha,i}^{n*} + 2
\overline{M}_{\alpha,i}) (\overline M_{\alpha,i}^{n*} -
\overline{M}_{\alpha,i})^2.
\end{eqnarray*}
Not that $R^t_{\alpha,i}$ is equal to the term $L_{\alpha,i}$ defined in the proof of Theorem~\ref{thm:entropy_st}. From the definition of the explicit part~\eqref{eq:kinsc1} 
and of the source terms $\delta \overline{M}_{\alpha,i-1/2+}$ \eqref{eq:defdeltaM} we can write
\begin{eqnarray}
R^t_{\alpha,i}  & = &
 \frac{g^2 \pi^2}{6} \sigma_i^2
(2\overline{M}_{\alpha,i} + \overline M^{n*}_{\alpha,i}) \left( \xi
  \overline{M}_{\alpha,i+1/2} - \xi\overline{M}_{\alpha,i-1/2} + \delta
  \overline{M}_{\alpha,i+1/2-} -  \delta \overline{M}_{\alpha,i-1/2+}
  \right)^2\nonumber\\
& \leq & \frac{2g^2 \pi^2}{3} \sigma_i^2
(2\overline{M}_{\alpha,i}+\overline M^{n*}_{\alpha,i}) \left(
         \xi^2\bigl( \overline M_{i+1/2+} - \overline M_{i+1/2-}
  \bigr)^2\1_{\xi <0} \right.\nonumber\\
& & \left. + \xi^2\bigl( \overline M_{i-1/2-} - \overline M_{i-1/2+}
  \bigr)^2\1_{\xi >0} +
u_i^2\bigl( \overline M_{i+1/2-} - \overline M_{i-1/2+}
\bigr)^2\right).
\label{eq:l1_1}
\end{eqnarray}
For the quantity $R^x_{\alpha,i}$, we first write $R^x_{\alpha,i} = R^x_{\alpha,i+} + R^x_{\alpha,i-}$ with
\begin{eqnarray*}
R^x_{\alpha,i+} & = &  \sigma_i\xi \left( H(\overline{M}_{\alpha,i+1/2},\xi,z_{b,i+1/2}) -
  \partial_1 H (\overline{M}_{\alpha,i},\xi,z_{b,i})(\overline{M}_{\alpha,i+1/2}-\overline{M}_{\alpha,i+1/2-})\right)\nonumber\\
& & - \sigma_i u_{\alpha,i}\partial_1 H (\overline{M}_{\alpha,i},\xi,z_{b,i})\overline{M}_{\alpha,i+1/2-} - \sigma_i\xi
H(\overline{M}_{\alpha,i},\xi,z_{b,i}),
\label{eq:L2+}
\end{eqnarray*}
and
\begin{eqnarray*}
R^x_{\alpha,i-} & = &  -\sigma_i\xi \left( H(\overline{M}_{\alpha,i-1/2},\xi,z_{b,i-1/2}) -
  \partial_1 H (\overline{M}_{\alpha,i},\xi,z_{b,i})(\overline{M}_{\alpha,i-1/2}-\overline{M}_{\alpha,i-1/2+})\right)\nonumber\\
& & + \sigma_i u_{\alpha,i}\partial_1 H (\overline{M}_{\alpha,i},\xi,z_{b,i})\overline{M}_{\alpha,i-1/2+}  + \sigma_i\xi H(\overline{M}_{\alpha,i},\xi,z_{b,i}),
\end{eqnarray*}
Let us rewrite $R^x_{\alpha,i+}$ under the form
\begin{eqnarray*}
R^x_{\alpha,i+} 
& = &  \sigma_i\xi \left( H(\overline{M}_{\alpha,i+1/2},\xi,z_{b,i+1/2}) -
  \partial_1 H (\overline{M}_{\alpha,i+1/2-},\xi,z_{b,i+1/2})(\overline{M}_{\alpha,i+1/2}-\overline{M}_{\alpha,i+1/2-})\right)\nonumber\\
& & -\sigma_i \xi (\partial_1 H (\overline{M}_{\alpha,i},\xi,z_{b,i}) - \partial_1 H (\overline{M}_{\alpha,i+1/2-},\xi,z_{b,i+1/2}))(\overline{M}_{\alpha,i+1/2}-\overline{M}_{\alpha,i+1/2-})\nonumber\\
& & - \sigma_i u_{\alpha,i}\partial_1 H (\overline{M}_{\alpha,i},\xi,z_{b,i})\overline{M}_{\alpha,i+1/2-}  - \sigma_i\xi H(\overline{M}_{\alpha,i},\xi,z_{b,i}).
\end{eqnarray*}
Using identity \eqref{eq:cubic} but for $\overline{M}_{\alpha,i+1/2+}$ and $\overline{M}_{\alpha,i+1/2-}$, 
we can obtained a relation similar to \eqref{eq:convleftbis0} that characterizes the linear dissipation associated to the scheme 
\begin{equation}\begin{array}{l}
	\dsp H(\overline{M}_{\alpha,i+1/2},\xi,z_{b,i+1/2}) = H(\overline{M}_{\alpha,i+1/2-},\xi,z_{b,i+1/2})\\
	\dsp\mkern 100mu
	+\partial_1 H(\overline{M}_{\alpha,i+1/2-},\xi,z_{b,i+1/2})(\overline{M}_{\alpha,i+1/2}-\overline{M}_{\alpha,i+1/2-})\\
	\dsp\mkern 100mu
+ \frac{g^2\pi^2}{6}(\overline{M}_{\alpha,i+1/2+} + 2\overline{M}_{\alpha,i+1/2-}) (\overline{M}_{\alpha,i+1/2} - \overline{M}_{\alpha,i+1/2-})^2.
	\label{eq:convleftbis}
	\end{array}
\end{equation}
Relation~\eqref{eq:convleftbis} allows then to write $R^x_{\alpha,i+}$ under the form
\begin{eqnarray*}
R^x_{\alpha,i+} & = & \sigma_i\xi \bigl( H(\overline{M}_{\alpha,i+1/2-},\xi,z_{b,i+1/2})
  +\frac{g^2\pi^2}{6}(\overline{M}_{\alpha,i+1/2+} + 2\overline{M}_{\alpha,i+1/2-}) (\overline{M}_{\alpha,i+1/2} -
  \overline{M}_{\alpha,i+1/2-})^2\bigr)\nonumber\\
& & -\sigma_i \xi (\partial_1 H (\overline{M}_{\alpha,i},\xi,z_{b,i}) -
\partial_1 H (\overline{M}_{\alpha,i+1/2-},\xi,z_{b,i+1/2}))(\overline{M}_{\alpha,i+1/2}-\overline{M}_{\alpha,i+1/2-})\nonumber\\
& & - \sigma_i u_{\alpha,i} \partial_1 H (\overline{M}_{\alpha,i},\xi,z_{b,i})\overline{M}_{\alpha,i+1/2-} -  \sigma_i\xi H(\overline{M}_{\alpha,i},\xi,z_{b,i}).
\end{eqnarray*}
Next, if $\overline{M}_{\alpha,i}(\xi)>0$, one has, refer to \eqref{eq:idH'}  
\begin{equation}
\partial_1 H(\overline{M}_{\alpha,i},\xi,z_{b,i})=\eta'(U_{\alpha,i})\kxi+gz_{b,i}
\label{eq:defhprime}
\end{equation}
and then
\begin{eqnarray}
	&&\partial_1 H(\overline{M}_{\alpha,i},\xi,z_{b,i})(\overline{M}_{\alpha,i+1/2+}-\overline{M}_{\alpha,i+1/2-})\nonumber
	\\
	&& \qquad 
	=\bigl(\eta'(U_{\alpha,i})\kxi+gz_{b,i}\bigr)(\overline{M}_{\alpha,i+1/2+}-\overline{M}_{\alpha,i+1/2-}),
	\label{eq:idprimeleft}
\end{eqnarray}
whereas, see \eqref{eq:convineqH_0},
\begin{eqnarray}
	&&H \partial_1 H (\overline{M}_{\alpha,i+1/2-},\xi,z_{b,i+1/2})(\overline{M}_{\alpha,i+1/2+}-\overline{M}_{\alpha,i+1/2-})\nonumber\\
	&& \qquad 
	\geq \bigl(\eta'(U_{\alpha,i+1/2-})\kxi+gz_{b,i+1/2}\bigr)(\overline{M}_{\alpha,i+1/2+}-\overline{M}_{\alpha,i+1/2-}).
	\label{eq:leftpos}
\end{eqnarray}
Taking the difference between \eqref{eq:leftpos} and \eqref{eq:idprimeleft}, we obtain
\begin{eqnarray}
	&&H \partial_1 H (\overline{M}_{\alpha,i+1/2-},\xi,z_{b,i+1/2})(\overline{M}_{\alpha,i+1/2+}-\overline{M}_{\alpha,i+1/2-})\nonumber\\
	&& \quad -\partial_1 H (\overline{M}_{\alpha,i},\xi,z_{b,i})(\overline{M}_{\alpha,i+1/2+}-\overline{M}_{\alpha,i+1/2-})\nonumber\\
	&& \qquad \geq l_\alpha\bigl(gh_{i+1/2-}-gh_{i}+gz_{b,i+1/2}-gz_{b,i}\bigr)(\overline{M}_{\alpha,i+1/2+}-\overline{M}_{\alpha,i+1/2-})\geq 0.
	\label{eq:keyleft}
\end{eqnarray}
From~\eqref{eq:keyleft}, it comes
\begin{eqnarray*}
R^x_{\alpha,i+}  & \leq &  \sigma_i\xi \bigl( H(\overline{M}_{\alpha,i+1/2-},\xi,z_{b,i+1/2})
  +\frac{g^2\pi^2}{6}(\overline{M}_{\alpha,i+1/2+} + 2\overline{M}_{\alpha,i+1/2-}) (\overline{M}_{\alpha,i+1/2} -
  \overline{M}_{\alpha,i+1/2-})^2\bigr)\nonumber\\
& & - \sigma_i u_{\alpha,i} \partial_1 H (\overline{M}_{\alpha,i},\xi,z_{b,i})\overline{M}_{\alpha,i+1/2-} -  \sigma_i\xi H(\overline{M}_{\alpha,i},\xi,z_{b,i})\nonumber\\
& = &  \sigma_i\xi \bigl( H(\overline{M}_{\alpha,i+1/2-},\xi,z_{b,i+1/2})
  +\frac{g^2\pi^2}{6}(\overline{M}_{\alpha,i+1/2+} + 2\overline{M}_{\alpha,i+1/2-}) (\overline{M}_{\alpha,i+1/2} -
  \overline{M}_{\alpha,i+1/2-})^2\bigr)\nonumber\\
& & - \sigma_i \partial_1 H (\overline{M}_{\alpha,i},\xi,z_{b,i})(\xi \overline{M}_{\alpha,i+1/2-} - \xi \overline{M}_{\alpha,i} + \delta
\overline{M}_{\alpha,i+1/2-}) -  \sigma_i\xi H(\overline{M}_{\alpha,i},\xi,z_{b,i}).
\end{eqnarray*}
Then, from \eqref{eq:defhprime}, we also get
\begin{eqnarray}
	&&\partial_1 H (\overline{M}_{\alpha,i},\xi,z_{b,i})
	\bigl(\xi \overline{M}_{\alpha,i+1/2-}-\xi \overline{M}_{\alpha,i}+\delta \overline{M}_{\alpha,i+1/2-}\bigr)\nonumber
	\\
	&& \qquad 
	=\bigl(\eta'(U_{\alpha,i})\kxi+gz_{b,i}\bigr)
	\bigl(\xi \overline{M}_{\alpha,i+1/2-}-\xi \overline{M}_{\alpha,i}+\delta \overline{M}_{\alpha,i+1/2-}\bigr)
	\label{eq:iidleft}
\end{eqnarray}
From~\eqref{eq:iidleft} it comes 
\begin{eqnarray}
R^x_{\alpha,i+} 
& \leq &  \sigma_i\xi \bigl( H(\overline{M}_{i+1/2-},\xi,z_{b,i+1/2})
  +\frac{g^2\pi^2}{6}(\overline{M}_{\alpha,i+1/2+} + 2\overline{M}_{\alpha,i+1/2-}) (\overline{M}_{\alpha,i+1/2} -
  \overline{M}_{\alpha,i+1/2-})^2\bigr)\nonumber\\
& & - \sigma_i \Bigl(\eta'(U_{\alpha,i})\kxi+gz_{b,i}\Bigr)(\xi \overline{M}_{\alpha,i+1/2-} - \xi \overline{M}_{\alpha,i} + \delta
\overline{M}_{\alpha,i+1/2-}) \nonumber\\
& & -  \sigma_i\xi H(\overline{M}_{\alpha,i},\xi,z_{b,i}).
\label{eq:estim_L2p}
\end{eqnarray}
An analoguous inequality can obviously be obtained for $R^x_{\alpha,i-} $ under the form
\begin{eqnarray}
R^x_{\alpha,i-} 
& \leq &  -\sigma_i\xi \bigl( H(\overline{M}_{\alpha,i-1/2+},\xi,z_{b,i-1/2})
  +\frac{g^2\pi^2}{6}(\overline{M}_{\alpha,i-1/2-} + 2\overline{M}_{\alpha,i-1/2+}) (\overline{M}_{\alpha,i-1/2} -
  \overline{M}_{\alpha,i-1/2+})^2\bigr)\nonumber\\
& & + \sigma_i \Bigl(\eta'(U_{\alpha,i})\kxi+gz_{b,i}\Bigr)(\xi \overline{M}_{\alpha,i-1/2+} - \xi \overline{M}_{\alpha,i} + \delta
\overline{M}_{\alpha,i-1/2+}) \nonumber\\
& & + \sigma_i\xi H(\overline{M}_{\alpha,i},\xi,z_{b,i}).
\label{eq:estim_L2m}
\end{eqnarray}

Adding the relation~\eqref{eq:entrfullystat2_JSM} to 
\eqref{eq:cindis01_d0} with the
estimates~\eqref{eq:l1_1},~\eqref{eq:estim_L2p},~\eqref{eq:estim_L2m} gives~\eqref{eq:entrfullystat3_JSM} proving the result.
\end{proof}

\begin{proof}[Proof of corollary~\ref{corol:entropy_discrete_estim}]
The proof of the result is similar to the one given by some of the
authors in ~\cite[Theorem~3.6]{JSM_entro}.
\end{proof}

\begin{proof}[Proof of corollary~\ref{corol:macro}]
As in the proof of Corollary~\ref{corol:macro_st}, it is possible to prove
that
$$\int_\R
(\overline{M}_{\alpha,i}^{n+1-})^3 d\xi = \int_\R
(\overline{M}_{\alpha,i}^{n+1})^3 d\xi,$$
and under the CFL condition~\eqref{eq:CFLfull} the quantity $d_{\alpha,i}^1$ is non
positive then
the sum for $\alpha=1,\ldots,N$ of relations~\eqref{eq:entrfullystat4_JSM} integrated in $\xi$ 
gives the result.
\end{proof}

\section{Fully discrete entropy inequality for the layer-averaged Navier-Stokes
  system}
\label{sec:lans}

The layer-averaging applied to the Euler system in
Section~\ref{sec:av_euler} can also be applied to the Navier-Stokes
system, see~\cite{BDGSM}. Considering a simplified Newtonian rheology,
the Navier-Stokes system~\eqref{eq:NS_2d1}-\eqref{eq:NS_2d3} can be written under the form
\begin{eqnarray}
& & \frac{\partial u}{\partial x}+\frac{\partial w}{\partial z} =0,\label{eq:NS_2d1_simple}\\
& & \frac{\partial u}{\partial t } + \frac{\partial u^2}{\partial x }
    +\frac{\partial uw }{\partial z }+ \frac{\partial p}{\partial x }
    = \mu\frac{\partial^2 u}{\partial x^2} + \mu\frac{\partial^2  u}{\partial z^2},\label{eq:NS_2d2_simple}\\
& & \hspace*{3cm} \frac{\partial p}{\partial z} = -g,
\label{eq:NS_2d3_simple}
\end{eqnarray}
where $\mu$ is a viscosity coefficient. The
system~\eqref{eq:NS_2d1_simple}-\eqref{eq:NS_2d3_simple} is completed
with the kinematic boundary conditions~\eqref{eq:free_surf},\eqref{eq:bottom} and suitable
dynamic boundary conditions.

Its layer-averaged version is given by
\begin{eqnarray}
& & \frac{\partial}{\partial t} h +
\frac{\partial}{\partial x} \sum_{j=1}^N h_\alpha u _\alpha = 0,\label{eq:NS_avz_111}\\
& & \frac{\partial}{\partial t} (h_\alpha u_\alpha) +
\frac{\partial}{\partial x}  \left( h_\alpha u_\alpha^2 + \frac{g}{2} h_\alpha h\right)
= - gh_{\alpha}\frac{\partial z_b}{\partial x} +
u_{\alpha+1/2}G_{\alpha+1/2} - u_{\alpha-1/2}
G_{\alpha-1/2}\nonumber\\
& & + \frac{\partial}{\partial x}\left(4\mu h_\alpha
  \frac{\partial u_\alpha}{\partial x}\right) + 2\mu\frac{u_{\alpha+1}
  - u_\alpha }{h_{\alpha+1}+h_\alpha} - 2\mu\frac{u_{\alpha} -
  u_{\alpha-1}}{h_{\alpha}+h_{\alpha-1}},\qquad
\alpha=2,\ldots,N-1\label{eq:NS_avz_222}\\
& & \frac{\partial}{\partial t} (h_1 u_1) +
\frac{\partial}{\partial x}  \left( h_1 u_1^2 + \frac{g}{2} h_1 h\right)
= -gh_1\frac{\partial z_b}{\partial x} +
u_{3/2}G_{3/2} \nonumber\\
& & \hspace*{6cm} + \frac{\partial}{\partial x}\left(4\mu h_1
  \frac{\partial u_1}{\partial x}\right) + 2\mu\frac{u_{2}
  - u_1 }{h_{2}+h_1} - \kappa u_1,\label{eq:NS_avz_222_1}\\
& & \frac{\partial}{\partial t} (h_N u_N) + \frac{\partial}{\partial
  x}  \left( h_N u_N^2 + \frac{g}{2} h_N h \right)
= -gh_N \frac{\partial z_b}{\partial x} - u_{N-1/2}G_{N-1/2}\nonumber\\
& & \hspace*{6cm}+ \frac{\partial}{\partial x}\left(4\mu h_N
  \frac{\partial u_N}{\partial x}\right) - 2\mu\frac{u_{N} -
  u_{N-1}}{h_{N}+h_{N-1}},\label{eq:NS_avz_222_N}\\
& & \frac{\partial}{\partial t} \left(\frac{z_{\alpha+1/2}^2 - z_{\alpha-1/2}^2}{2}\right) +
\frac{\partial}{\partial x} \left( \frac{z_{\alpha+1/2}^2 -
    z_{\alpha-1/2}^2}{2}u_\alpha\right)  dz = h_\alpha
w_\alpha\nonumber\\
& & \hspace*{6cm} + z_{\alpha+1/2}G_{\alpha+1/2} - z_{\alpha-1/2}
G_{\alpha-1/2},\label{eq:NS_avz_444}
\end{eqnarray}
where $\kappa$ is a Navier friction coefficient at the bottom and the exchange terms $G_{\alpha+1/2}$ are given by~\eqref{eq:Qalphabis}-\eqref{eq:Qlim}.

The system~\eqref{eq:NS_avz_111}-\eqref{eq:NS_avz_444} is rewritten
under the compact form
\begin{equation}
\frac{\partial U}{\partial t} + \frac{\partial F(U)}{\partial x} =
S_e(U,\partial_t U,\partial_x U) + S_b(U) + S_{v,f}(U),
\label{eq:glo}
\end{equation}
with $U$ defined by~\eqref{eq:stateU}.

We denote with $F(U)$ the flux of the conservative part,
and with $S_e(U,\partial_t U,\partial_x U)$, $S_b(U)$ and $S_{v,f}(U)$ the source terms,
 representing respectively the
mass transfer, the topography, and the viscous and
friction effects.

For the time discretization, we apply a time splitting technique to the equations 
(\ref{eq:glo}) and we write
\begin{eqnarray}
&&\frac{{\tilde U}^{n+1}-U^{n}}{\Delta t^n} + \frac{\partial
  F(U^n)}{\partial x} = S_e(U^n,{\tilde U}^{n+1}) +  S_b(U^n),
\label{eq:glo1}\\
&&\frac{U^{n+1}-{\tilde U}^{n+1}}{\Delta t^n} - S_{v,f} (U^{n},U^{n+1})=0.
\label{eq:glo2}
\end{eqnarray}
Equation~\eqref{eq:glo1} corresponds to the semi-discrete in time
version of the layer-averaged Euler
system~\eqref{eq:HH}-\eqref{eq:eq44} whose discretization has been studied
in Section~\ref{sec:properties}. It remains to propose a discretization for Eq.~\eqref{eq:glo2}.
%
%
%
Since the viscous and friction terms $S_{v,f}$ in
(\ref{eq:glo2}) are
dissipative, they are treated via a semi-implicit scheme for stability
reasons. By using a finite differences discretization in space, this leads to
\begin{eqnarray}
(h u)^{n+1}_{\alpha,i} & = & (\widetilde{hu})^{n+1}_{\alpha,i} +
\frac{8\mu}{\Delta x_i} \left(
  h_{\alpha,i+1/2}\frac{u^{\frac{1}{2}}_{\alpha,i+1} - u^{\frac{1}{2}}_{\alpha,i}}{\Delta x_{i+1} + \Delta x_i} - 
 h_{\alpha,i-1/2} \frac{u^{\frac{1}{2}}_{\alpha,i} - u^{\frac{1}{2}}_{\alpha,i-1}}{\Delta x_{i} + \Delta x_{i-1}} \right)\nonumber\\
& & + 2\mu\frac{u^{\frac{1}{2}}_{\alpha+1,i}
  - u^{\frac{1}{2}}_{\alpha,i} }{h_{\alpha+1,i}+h_{\alpha,i}} - 2\mu\frac{u^{\frac{1}{2}}_{\alpha,i} -
  u^{\frac{1}{2}}_{\alpha-1,i}}{h_{\alpha,i}+h_{\alpha-1,i}} - \kappa \delta_{1,\alpha} u^{\frac{1}{2}}_{\alpha,i},\label{eq:glo2_dis}
\end{eqnarray}
for $\alpha=1,\ldots,N$. The superscript $y^{\frac{1}{2}}$ means
$y^{\frac{1}{2}} = \frac{y^n+ {\tilde y}^{n+1}}{2}$ and
$\delta_{1,\alpha}$ is the Kronecker symbol.

The following proposition holds.
\begin{proposition}
Following~\eqref{eq:kinmomE}, $E_{\alpha,i}$ is given by
$$E_{\alpha,i} = \frac{h_{\alpha,i}}{2}(u_{\alpha,i})^2 +
\frac{g}{2}h_{\alpha,i}h_i + gz_{b,i}h_{\alpha,i},$$
and with the notations of theorem~\ref{thm:entropy_discrete}, we have
\begin{eqnarray*}
	E_{\alpha,i}^{n+1} & \leq & E_{\alpha,i} -\sigma_i\Bigl(\int_{\R}
        \widetilde{H}_{\alpha,i+1/2-} d\xi -
        \int_{\R}\widetilde{H}_{\alpha,i-1/2+} d\xi\Bigr) \nonumber\\
& & -\Delta t^n\Bigl( \int_{\R}\widehat{H}_{\alpha+1/2,i}^{n+1-} d\xi-
        \int_{\R}\widehat{H}_{\alpha-1/2,i}^{n+1-} d\xi\Bigr)\nonumber\\
& & + \int_{\R}
        d_{\alpha,i} d\xi+ \int_{\R} e_{\alpha,i} d\xi + \Delta t^n
        {\cal D}_{\alpha,i},
	\label{eq:entropy_ns}
\end{eqnarray*}
with
\begin{eqnarray*}
{\cal D}_{\alpha,i} & = &
\frac{8\mu}{\Delta x_i} \left(
  h_{\alpha,i+1/2}\frac{u^{\frac{1}{2}}_{\alpha,i+1} +
    u^{\frac{1}{2}}_{\alpha,i}}{2}\frac{u^{\frac{1}{2}}_{\alpha,i+1} -
    u^{\frac{1}{2}}_{\alpha,i}}{\Delta x_{i+1} + \Delta x_i}
\right.\nonumber\\
& & \left. - 
 h_{\alpha,i-1/2} \frac{u^{\frac{1}{2}}_{\alpha,i} + u^{\frac{1}{2}}_{\alpha,i-1}}{2}\frac{u^{\frac{1}{2}}_{\alpha,i} - u^{\frac{1}{2}}_{\alpha,i-1}}{\Delta x_{i} + \Delta x_{i-1}} \right)\\
& & + 2\mu\frac{u^{\frac{1}{2}}_{\alpha+1,i} + u^{\frac{1}{2}}_{\alpha,i}}{2}\frac{u^{\frac{1}{2}}_{\alpha+1,i}
  - u^{\frac{1}{2}}_{\alpha,i} }{h_{\alpha+1,i}+h_{\alpha,i}} - 2\mu\frac{u^{\frac{1}{2}}_{\alpha,i} + u^{\frac{1}{2}}_{\alpha-1,i}}{2}\frac{u^{\frac{1}{2}}_{\alpha,i} -
  u^{\frac{1}{2}}_{\alpha-1,i}}{h_{\alpha,i}+h_{\alpha-1,i}}
\nonumber\\
& & -
\frac{4\mu}{\Delta x_i} \left(
  h_{\alpha,i+1/2}\frac{\left( u^{\frac{1}{2}}_{\alpha,i+1} -
    u^{\frac{1}{2}}_{\alpha,i}\right)^2}{\Delta x_{i+1} + \Delta x_i} - 
 h_{\alpha,i-1/2} \frac{\left(u^{\frac{1}{2}}_{\alpha,i} - u^{\frac{1}{2}}_{\alpha,i-1}\right)^2}{\Delta x_{i} + \Delta x_{i-1}} \right)\\
& & - \mu\frac{\left(u^{\frac{1}{2}}_{\alpha+1,i} -
    u^{\frac{1}{2}}_{\alpha,i}\right)^2}{h_{\alpha+1,i}+h_{\alpha,i}}
- \mu\frac{\left( u^{\frac{1}{2}}_{\alpha,i} - u^{\frac{1}{2}}_{\alpha-1,i}\right)^2}{h_{\alpha,i}+h_{\alpha-1,i}}- \kappa \delta_{1,\alpha} (u^{\frac{1}{2}}_{\alpha,i})^2,
\end{eqnarray*}
\label{prop:entro_ns}
\end{proposition}
\begin{proof}[Proof of prop.~\ref{prop:entro_ns}]
First we notice that the semi-implicit step does not modify the water
depth i.e. $h^{n+1}_i={\tilde h}^{n+1}_i$. Multiplying~\eqref{eq:glo2_dis}
by $u_{\alpha,i}^\frac{1}{2}$, it comes
\begin{equation}
\frac{h^{n+1}_{\alpha,i}}{2} (u^{n+1}_{\alpha,i})^2 = 
\frac{{\tilde h}^{n+1}_{\alpha,i}}{2} ({\tilde u}^{n+1}_{\alpha,i})^2
+ \Delta t^n {\cal D}_{\alpha,i},
\label{eq:entropy_ns12}
\end{equation}
and ${\cal D}_{\alpha,i}$ is a consistent discretization of the viscous and friction terms
appearing in~\eqref{eq:energy_eq}. The sum of relation~\eqref{eq:entrfullystat3_JSM}
integrated in $\xi$ over $\R$ and relation~\eqref{eq:entropy_ns12}
gives the result.
\end{proof}

\section*{Acknowledgements}

The authors wish to express their warm thanks to Fran\c{c}ois Bouchut for many fruitful discussions. 

\bibliographystyle{amsplain}
\bibliography{boussinesq}

\end{document}